\documentclass[12pt]{amsart}

\usepackage[margin=1in]{geometry}  
\usepackage{graphicx}              
\usepackage{amsmath}               
\usepackage{amsfonts}              
\usepackage{amsthm}                
\usepackage{amsmath, amssymb}
\usepackage[all]{xy}
\usepackage{enumerate}
\usepackage{graphicx}

\newtheorem{thm}{Theorem}[section]
\newtheorem{lem}[thm]{Lemma}
\newtheorem{prop}[thm]{Proposition}
\newtheorem{cor}[thm]{Corollary}

\newenvironment{remark}[1][Remark.]{\begin{trivlist}
\item[\hskip \labelsep {\bfseries #1}]}{\end{trivlist}}

\theoremstyle{definition}
\newtheorem{defn}[thm]{Definition}

\theoremstyle{definition}
\newtheorem{exmp}[thm]{Example}

\begin{document}

\nocite{*}

\title{Injective Presentations of Induced Modules over Cluster-Tilted Algebras}

\author{Ralf Schiffler}   
\address{Department of Mathematics, University of Connecticut, 
Storrs, CT 06269-3009, USA}
\email{schiffler@math.uconn.edu}
\author{Khrystyna Serhiyenko}\thanks{The authors were supported by the NSF CAREER grant DMS-1254567 and by the University of Connecticut. The second author was also supported by the NSF Postdoctoral fellowship MSPRF-1502881.}
\address{Department of Mathematics, University of California, Berkeley, 
CA 94720-3840, USA}
\email{khrystyna.serhiyenko@berkeley.edu}

\maketitle

\begin{abstract} 
Every cluster-tilted algebra $B$ is the relation extension $C\ltimes \textup{Ext}^2_C(DC,C)$ of a tilted algebra $C$. A $B$-module is called induced if it is of the form $M\otimes_C B$ for some $C$-module $M$. We study the relation between the injective presentations of a $C$-module and the injective presentations of the induced $B$-module. Our main result is an explicit construction of the modules and morphisms in an injective presentation of any induced $B$-module.  In the case where the $C$-module, and hence the $B$-module, is projective, our construction yields an injective resolution. In particular, it gives a module theoretic proof of the well-known 1-Gorenstein property of cluster-tilted algebras.

\end{abstract}

\section{Introduction}
Cluster-tilted algebras are finite dimensional associative algebras which were
introduced in \cite{BMR} and, independently, in \cite{CCS} for the type $\mathbb{A}$. 

One motivation for introducing these algebras came from Fomin and Zelevinsky's cluster algebras \cite{FZ}. To every cluster in an acyclic cluster algebra one can associate a cluster-tilted algebra, and the indecomposable rigid modules over the cluster-tilted algebra correspond bijectively to the cluster variables outside the chosen cluster. Generalizations of cluster-tilted algebras, the Jacobian algebras of quivers with potentials, were introduced in \cite{DWZ}, extending this correspondence to the non-acyclic types. 
Many people have studied cluster-tilted algebras in this context, see for example \cite{BBT,BMR, BMR2, BMR3, CCS2, CC, CK, KR}.

The second motivation came from classical tilting theory. Tilted algebras are the endomorphism algebras of tilting modules over hereditary algebras, whereas cluster-tilted algebras are the endomorphism algebras of cluster-tilting objects over  cluster categories of  hereditary algebras. This similarity in the two definitions lead to the following precise relation between tilted and cluster-tilted algebras, which was established in \cite{ABS}.

There is a surjective map
\[\xymatrix{ \{\textup{tilted algebras}\} \ar@{->>}[r] &\{\textup{cluster-tilted algebras}\} ,& C\ar@{|->}[r]&B=C\ltimes E,
}\]
where $E$ denotes the $C$-$C$-bimodule $E=\text{Ext}^2_C(DC,C)$ and $C\ltimes E$ is the trivial extension.

This result allows one to define cluster-tilted algebras without using the cluster category.
It is natural to ask how the module categories of $C$ and $B$ are related, and several results in this direction have been obtained, see for example \cite{ABS2,ABS3,ABS4, BFPPT,BOW,DS}. 

The Hochschild cohomology of the algebras $C$ and $B$ has been compared in \cite{AR,ARS,ABIS,L,AGST}.

\bigskip
In \cite{SS}, we initiated a new approach to study the relation between the module categories of a tilted algebra $C$ and its cluster-tilted algebra $B=C\ltimes E$, namely \emph{induction} and \emph{coinduction}.

The induction functor
$-\otimes_C B$ and the coinduction functor $\text{Hom}_C(B,-)$ from $\textup{mod}\,C$ to $\textup{mod}\,B$ are defined whenever $C$ is a subring of $B$ which has the same identity. If we are dealing with algebras over a field $k$, we can, and usually do, write the coinduction functor as $D(B\otimes_CD-)$, where $D=\text{Hom}(-,k)$ is the standard duality.

Induction and coinduction are important tools in classical Representation Theory of Finite Groups. In this case, $B$ would be the group algebra of a finite group $G$ and $C$ the group algebra of a subgroup of $G$ (over a field whose characteristic is not dividing the group orders). In this situation, the algebras are semi-simple, induction and coinduction are the same functor, and this functor is exact.

For arbitrary rings, and even for finite dimensional algebras, the situation is not that simple. In general, induction and coinduction are not the same functor and, since the $C$-module $B$ is not projective (and not flat), induction and coinduction are not exact functors.

However, the connection between tilted algebras and cluster-tilted algebras is close enough so that induction and coinduction are interesting tools for the study of the relation between the module categories.

\smallskip

 In this paper, we use induction and coinduction to construct explicit injective presentations of induced modules over cluster-tilted algebras.
Since the induction functor sends projective $C$-modules $P_C$ to projective $B$-modules $P_B=P_C\otimes_CB$, and the coinduction functor sends injective $C$-modules $I_C$ to injective $B$-modules $I_B=D(B\otimes_CDI_C)$, we are able to construct injective presentations in $\textup{mod}\,B$ from corresponding injective presentations in $\textup{mod}\,C$. 

Our main result is the following.  Here $\nu$ denotes the Nakayama functor and $\Omega$ the first syzygy.

\begin{thm}\label{main thm}
Let $C$ be a tilted algebra, $B$ the corresponding cluster-tilted algebra, $M$ an \emph{indecomposable} $C$-module. Let 
$$\xymatrix {0\ar[r]& M \ar[r]&I^0_C \ar[r]&I^1_C&\text{and}& 0\ar[r]& M\otimes_C E \ar[r]&\bar{I}^0_C \ar[r]&\bar{I}^1_C}$$
be minimal injective presentations in {\upshape{mod}}$\, C$, 
$$\xymatrix{0\ar[r] & 
{\tau\Omega M}\ar[r] & \hat{I}_C}$$
an injective envelope in \upshape{mod}$\,C$, and let $\tilde{I}_C$ be the injective $C$-module $\tilde{I}_C=\nu\nu^{-1}\Omega^{-1} M$.  Then 
$$\xymatrix {0\ar[r]&M\otimes_C B \ar[r]&I^0_B\oplus\bar{I}^0_B \ar[r]&\tilde{I}_B\oplus\bar{I}^1_B\oplus \hat{I}_B}$$
is an injective presentation of $M\otimes_C B$ in {\upshape{mod}}$\,B$. 
\end{thm}

In the special case where $M$ is a projective $C$-module, we can show that the presentation is actually a resolution, and we obtain the following result.

\begin{cor}\label{cor intro}  With the notation above,
 if $M= P_C$ is a projective $C$-module, then  
 $$\xymatrix {0\ar[r]& P_C\otimes_C B \ar[r]&I^0_B\oplus\bar{I}^0_B \ar[r]&\tilde{I}_B\oplus\bar{I}^1_B\ar[r]&0}$$
 is an injective resolution in $\textup{mod}\,B$.
\end{cor}

%
A dual version of each of the two results also holds, and in the case where the projective dimension of $M$ is at most 1, we also compute the cokernel of the last map in the presentation.

As an immediate consequence, we obtain a new proof, which does not use cluster categories, of a result by Keller and Reiten \cite{KR}.
\begin{cor}\label{corKR}
 Cluster-tilted algebras are 1-Gorenstein.
\end{cor}
The proof of the theorems uses both induction and coinduction, and it relies greatly on the particular structure of the bimodule $E$. 

The paper is organized as follows. In section \ref{sect:basics}, we set up the notation and recall results about induction and coinduction functors from \cite{SS}. Section \ref{sect 5} is devoted to the proofs of the main results. In section \ref{sect ex}, we give several examples of injective presentations and describe a method for constructing injective resolutions.

\smallskip

\paragraph{\emph{ Historical remark.}} {Corollaries \ref{cor intro} and \ref{corKR} were originally part of the first preprint version of \cite{SS}. Following the suggestion of an anonymous referee, we decided to remove these results from \cite{SS}.

\section{Notation and Preliminaries}
\label{sect:basics}
Throughout this paper all algebras are assumed to be basic, finite dimensional over an algebraically closed field $k$.  Suppose $Q=(Q_0, Q_1)$ is a connected quiver without oriented cycles.  By $kQ$ we denote the path algebra of $Q$.  If $\Lambda$ is a $k$-algebra then denote by mod$\,\Lambda$  the category of finitely generated right $\Lambda$-modules and by ind$\,\Lambda$ a set of representatives of each isoclass of indecomposable right $\Lambda$-modules.  Given $M \in $ mod$\,\Lambda$, the projective dimension of $M$ in mod$\,\Lambda$ is denoted by pd$_\Lambda M$ and its injective dimension by id$_\Lambda M$. Let $\tau$ be the Auslander-Reiten translation.  Also define $\Omega M$ to be the first syzygy and $\Omega^{-1} M$ the first cosyzygy of $M$.  By $D$ we understand the standard duality functor Hom$_k(-,k)$.  Finally, let $\nu=D\text{Hom}_\Lambda(-,\Lambda)$ be the Nakayama functor and $\nu ^{-1}=\text{Hom}_\Lambda(D\Lambda,-)$ be the inverse Nakayama functor. For further details on representation theory we refer to \cite{ASS, S}. 

\subsection{Tilted algebras and cluster-tilted algebras}
 
 Let $A=kQ$ be a hereditary algebra.  We recall, that an $A$-module $T$ is called \emph{tilting} if Ext$_A^1 (T,T)=0$ and the number of indecomposable direct summands of $T$ equals the number of isomorphism classes of simple $A$-modules. The corresponding algebra $C=\text{End}_A T$ is called a \emph{tilted algebra}.  

Let    $\mathcal{D}=\mathcal{D}^b(\text{mod}\,A)$ denote the derived category of bounded complexes  of $A$-modules. The \emph{cluster category} $\mathcal{C}_A$ is defined as the orbit category of the derived category with respect to the functor $\tau_{\mathcal{D}}^{-1}[1]$, where $\tau_{\mathcal{D}}$ is the Auslander-Reiten translation in the derived category and $[1]$ is the shift.  Cluster categories were introduced in \cite{BMRRT}, and in \cite{CCS} for type $\mathbb{A}$, and were further studied in \cite{K,KR,A,P}.  They are triangulated categories \cite{K}, that  have Serre duality and are 2-Calabi Yau \cite{BMRRT}.  

An object $T$ in $\mathcal{C}_A$ is called \emph{cluster-tilting} if $\textup{Ext}^1_{\mathcal{C}_A}(T,T)=0$ and the number of indecomposable direct summands of $T$ equals the number of isomorphism classes of simple $A$-modules. The endomorphism algebra $\textup{End}_{\mathcal{C}_A}T$ of a cluster-tilting object is called a \emph{cluster-tilted algebra} \cite{BMR}. 
%

\subsection{Relation extensions}
Let $C$ be an algebra of global dimension at most two and let $E$ be the $C$-$C$-bimodule $E=\text{Ext}_C ^2 (DC,C)$.
The \emph{relation extension} of $C$ is the trivial extension algebra $B=C\ltimes E$, whose underlying $C$-module is $C\oplus E$, and multiplication is given by $(c,e)(c',e')=(cc',ce'+ec')$.  
Relation extensions where introduced in \cite{ABS}. In the special case where $C$ is a tilted algebra, we have the following result.

\begin{thm}\cite{ABS}
Let $C$ be a tilted algebra. Then $B=C\ltimes \textup{Ext}_C ^2 (DC,C)$ is a cluster-tilted algebra. Moreover all cluster-tilted algebras are of this form.
\end{thm}

 \begin{remark}
This shows that the tilted algebra $C$ is a subalgebra and a quotient of the cluster-tilted algebra $B$.  Let $\pi: B\to C$ be the corresponding surjective algebra homomorphism.  We will often consider a $C$-module $M$ as a $B$-module with action $M\cdot b = M\cdot \pi(b)$.  In particular, $C$ and $E$ are right $B$-modules, and there exists a short exact sequence in mod$\,B$
\begin{equation}\label{(1)}    \xymatrix{0\ar[r]&E\ar[r]^i&B\ar[r]^{\pi}&C\ar[r]&0}. \end{equation}
\end{remark}

\subsection{Induction and Coinduction Functors}\label{sect 3}

In this section we define two functors called induction and coinduction and describe some general results about them. For proofs we refer to \cite{SS}.  Suppose there are two $k$-algebras $C$ and $B$ with the property that $C$ is a subalgebra of $B$ and they share the same identity.  Then there is a general construction via the tensor product, also known as \emph{extension of scalars}, that sends a $C$-module to a particular $B$-module.  We give a precise definition below.
\begin{defn}
Let $C$ be a subalgebra of $B$, such that $1_C=1_B$, then 
$$-\otimes _C B:  \:  \text{mod}\,C \rightarrow \text{mod}\,B$$ 
is called the {\em{induction functor}}, and dually  
$$D(B\otimes_C D-): \: \text{mod}\,C \rightarrow \text{mod}\,B$$ 
is called the {\em{coinduction functor}}.   Moreover, given $M\in$ mod$\,C$ the corresponding {\em{induced module}} is defined to be $M\otimes _C B$, and the {\em{coinduced module}} is defined to be $D(B\otimes _C DM)$.  
\end{defn}
 
First observe that both functors are covariant.  The induction functor is right exact, while the coinduction functor is left exact. 

\begin{lem} \label{*}
Let $C$ and $B$ be two $k$-algebras and $N$ a $C$-$B$-bimodule, then {\upshape{
$$M\otimes _ C N \cong D \text{Hom} _C ( M, DN)$$}}
 as $B$-modules for all {\upshape{$M \in \text{mod} \, C$}}. 
\end{lem}  
\begin{proof}
$M\otimes_C N \cong D \text{Hom}_k( M\otimes_C N, k)\cong D\text{Hom}_C(M, \text{Hom}_k (N, k))\cong D\text{Hom}_C (M, DN)$.  
\end{proof}

The next proposition describes an alternative definition of these functors, and we will use these two descriptions interchangeably.
 
\begin{prop}\label{3.3}
Let $C$ be a subalgebra of $B$ such that $1_C=1_B$, then  for every $M\in$ {\upshape{mod}}$\,C$\\
\upshape{ \indent (a) $M\otimes_C B\cong D\text{Hom}_C(M,DB)$.\\
\indent (b) $D(B\otimes_C DM)\cong \text{Hom}_C(B, M)$.}
\end{prop}

We shall need the following basic properties of these functors. 

\begin{prop} \label{3.4}
Let $C$ be a subalgebra of $B$ such that $1_C=1_B$.  If $e$ is an idempotent then\\
\indent {\upshape {(a)}} $(eC)\otimes _C B\cong eB$.\\
\indent{\upshape{(b)}} $D(B\otimes_C Ce)\cong DBe$. \\
In particular, if $P(i)$ and $I(i)$ are indecomposable projective and injective $C$-modules at vertex $i$, then $P(i)\otimes_C B$ and $D(B\otimes _C D I(i))$ are respectively indecomposable projective and injective $B$-modules at vertex $i$.  
\end{prop}

We can say more in the situation when $B$ is  a split extension of $C$.
  
\begin{defn}
Let $B$ and $C$ be two algebras.  We say {\em{$B$ is a split extension of $C$ by a nilpotent bimodule $E$}} if there exist algebra homomorphisms $\pi: B\to C$ and $\sigma: C \to B$ such that $\pi\sigma$ is the identity on $C$ and $E=\text{ker}\,\pi$ is a nilpotent (two-sided) ideal of $B$.   

In particular, there exists the following short exact sequence of $B$-modules. 
\begin{equation}\label{2}    \xymatrix{0\ar[r]&E\ar[r]^{i}&B \ar@<.5ex>[r]^{\pi}&C \ar@<.5ex>[l]^{\sigma}\ar[r]&0} \end{equation}
\end{defn}

\noindent For example, relation extensions are split extensions.

If $B$ is a split extension of $C$ then $\sigma$ is injective, which means $C$ is a subalgebra of $B$.  Also, $E$ is a $C$-$C$-bimodule, and we require $E$ to be nilpotent so that $1_B=1_C$.   Observe that $B\cong C\oplus E$ as $C$-modules, and there is an isomorphism of $C$-modules $M\otimes_C B\cong M\otimes _C C\oplus M\otimes _C E\cong M\oplus M\otimes_C E$.  Similarly, $D(B\otimes_C DM)\cong M\oplus D(E\otimes_C DM)$ as $C$-modules.  This shows that induction and coinduction of a module $M$ yields the same module $M$ plus possibly something else.  The next proposition shows a precise relationship between a given $C$-module and its image under the induction and coinduction functors. 

\begin{prop} \label{3.6}
Suppose $B$ is a split extension of $C$ by a nilpotent bimodule $E$, then for every  $M\in$ {\upshape{mod}}$\,C$ there exist two short exact sequences of $B$-modules:\\
\indent {\upshape{(a)}} \xymatrix{0\ar[r]&M\otimes_C E\ar[r]&M\otimes_C B\ar[r]& M\ar[r]&0}.\\
\indent {\upshape{(b)}} \xymatrix{0\ar[r]&M\ar[r]&D(B\otimes_C DM)\ar[r]& D(E\otimes_C DM)\ar[r]&0}.
\end{prop}
Thus, in this situation each module is a quotient of its induced module and a submodule of its coinduced module.

\subsection{ Induced and Coinduced Modules over Relation Extension Algebras}\label{sect 4}

In this subsection, we recall several results from \cite{SS} on the  properties of the induction and coinduction functors in the case  where $C$ is an algebra of global dimension at most two and $B$ is its relation extension. 
Throughout this subsection, 
all tensor products $\otimes$ are tensor products over $C$.

 We begin by establishing some properties of $E$ and $DE$.

\begin{prop} \label{4.1}\cite{SS}
Let $C$ be an algebra of global dimension at most 2. Then \\
{\upshape{\indent (a) $E\cong \tau^{-1}\Omega^{-1} C $.\\
\indent (b) $DE\cong \tau\Omega DC$.\\
\indent (c) $M\otimes E\cong \tau^{-1}\Omega^{-1} M$.\\
\indent (d) $D(E\otimes DM)\cong \tau\Omega M$.}}
\end{prop}

The following result gives  homological conditions under which induction or coinduction is trivial.
\begin{prop}\label{4.3}\cite{SS}
Let $C$ be an algebra of global dimension at most 2, and let $B=C\ltimes E$.  Suppose $M\in$ {\upshape{mod}}$\,C$,  then \\
\indent{\upshape{(a)}} {\upshape{id}}$_C M\leq 1$ if and only if $M\otimes B\cong M$.\\
\indent{\upshape{(b)}} {\upshape{pd}}$_C M\leq 1$ if and only if $D(B\otimes DM)\cong M$.
\end{prop}

The next lemma is a slight modification of a result of \cite{SS}, and we include a proof for completeness.
\begin{lem}\label{4.6}\cite{SS}
Let $C$ be an algebra of global dimension at most 2,  then for any $M\in\textup{mod}\,C$\\
\indent {\upshape{(a) Ext$_C^1(M\otimes E,C)=0$.\\
\indent (b) Ext$_C^1(DC, D(E\otimes DM) )=0$.\\
\indent (c) Ext$_C^1 (E, M\otimes E)=0$.\\       
\indent (d) Ext$_C^1 (D(E\otimes DM),DE)=0$.}}
\end{lem}

\begin{proof}
We will show parts (a) and (c), and the rest of the lemma can be proven similarly.  \\
Part (a).  By Proposition \ref{4.1}(c), we see that Ext$^1_C (M\otimes E,C)\cong \text{Ext}^1_C(\tau^{-1}\Omega^{-1} M, C)$, which in turn by the Auslander-Reiten formula is isomorphic to $D\overline{\text{Hom}}_C (C,\Omega^{-1} M)$.  Let $i: M\rightarrow I$ be an injective envelope of $M$, thus we have the following short exact sequence 
$$ \xymatrix{0\ar[r]&M\ar[r]^i &I\ar[r]^-{\pi}&\Omega^{-1}M\ar[r]&0}. $$
Applying Hom$_C(C,-)$ to this sequence we obtain an exact sequence
 $$\xymatrix{0\ar[r]& \text{Hom}_C(C,M)\ar[r]&\text{Hom}_C (C,I)\ar[r]^-{\pi_*}&\text{Hom}_C(C,\Omega^{-1}M)\ar[r]&\text{Ext}^1_C(C,M)}.$$ 
However,  Ext$^1_C(C,M)=0$ shows that $\pi_*$ is surjective.  This implies that every morphism from $C$ to $\Omega^{-1}M$ factors through the injective $I$. Thus,  $\overline{\text{Hom}}_C(C,\Omega^{-1}M)=0$, and this shows part (a). \\
Part (c).  As above observe that Ext$^1_C(E,M\otimes E)\cong D \overline{\text{Hom}}_C(M\otimes E,\Omega^{-1}C)$.  Let $j: C\rightarrow J$ be an injective envelope of $C$, thus we have the following short exact sequence 
$$ \xymatrix{0\ar[r]&C\ar[r]^j &J\ar[r]^-{\rho}&\Omega^{-1}C\ar[r]&0}, $$
Applying Hom$_C(M\otimes E,-)$ to this sequence we get  an exact sequence
 $$\xymatrix{0\ar[r]& \text{Hom}_C(M\otimes E,C)\ar[r]&\text{Hom}_C (M\otimes E,J)\ar[r]^-{\rho_{*}}&\text{Hom}_C(M\otimes E,\Omega^{-1}C)\ar[r]&{0}}.$$
where the surjectivity of $\rho_*$ follows from part (a).
Thus $\overline{\text{Hom}}_C(M\otimes E,\Omega^{-1}C)=0$, and this completes the proof of part (c).
\end{proof}   

Recall that a module $M$ is called rigid if $\textup{Ext}^1 (M,M)=0$.
\begin{cor}\label{4.7}\cite{SS}
If the global dimension of $C$ is at most two, then both $E\oplus C$ and $DE\oplus DC$ are rigid modules. 
\end{cor}

The next result only holds for tilted algebras. 

\begin{lem} \label{4.4}\cite{SS}
Let $C$ be a tilted algebra. Then for all {\upshape{$M\in \text{mod} \,C$}} \\
\indent {\upshape{(a)}}  {\upshape{id}}$_C M\otimes E \leq 1$.\\
\indent {\upshape{(b)}}  {\upshape{pd}}$_C D(E\otimes DM) \leq 1$.
\end{lem}

\section{Main results}\label{sect 5}

In this section, we construct an explicit injective presentation of an arbitrary induced   module $M\otimes B$ in a cluster-tilted algebra $B=C\ltimes E$ using only induction and coinduction functors applied to modules over a tilted algebra $C$.  This presentation is described completely in terms of $C$-modules. 
In particular, our method provides an injective {\em resolution} for any projective $B$-module.
  We start by showing preliminary results that lead to the main theorem.  Most of these statements are also true when the global dimension of $C$ is at most 2 and we make that distinction clear.

\begin{lem}\label{5.1}
Let {\upshape{gl.dim}}$\,C = 2$ and $M\in$ {\upshape{mod}}$\,C$.  Suppose that {\upshape{id}}$_C M=2$, and that 

$$\begin{array}{cc}
\xymatrix{0\ar[r]&M\ar[r]^{i_0}&I_C^0\ar[r]^{i_2}&I_C^1\ar[r]^{i_3}&I_C^2\ar[r]&0}\end{array}$$
is a minimal injective resolution of $M$.  Then there is a commutative diagram with exact rows  

$$
\xymatrix {0\ar[r]&M\ar[r]^-{i_0}\ar[d]^-{g_1}&I_C^0\ar[r]^-{i_1}\ar[d]^{g_0}&\Omega^{-1}M\ar[r]\ar@{=}[d]&0\\
0\ar[r]&\tau\Omega\tau^{-1}\Omega^{-1}M\ar[r]^-{\pi_0}&\tilde{I}_C\ar[r]^-{\pi_1}&\Omega^{-1}M\ar[r]&0}
$$
where $\tilde{I}_C=\nu\nu^{-1}\Omega^{-1}M$ is an injective $C$-module.  
\end{lem}

\begin{proof}
Consider the diagram below.  The injective resolution of $M$ also gives a minimal injective resolution of $\Omega^{-1}M$, which is shown in the top row of the diagram.  Then we apply the inverse Nakayama functor $\nu^{-1}=\text{Hom}_C(DC, - )$ to this resolution and obtain a projective presentation of $\tau^{-1}\Omega^{-1}M$ in the second row.  First, note that $\nu^{-1}\Omega^{-1}M=\text{Hom}_C(DC,\Omega^{-1}M)$ is nonzero, because it contains the nonzero map $i_1$.  Also, the global dimension of $C$ is two, which means that $\nu^{-1}\Omega^{-1}M$ is projective and we actually have a projective resolution of $\tau^{-1}\Omega^{-1}M$.

Next we apply the Nakayama functor $\nu=D\text{Hom}_C(-,C)$ to the projective resolution of $\Omega\tau^{-1}\Omega^{-1} M$, and obtain an injective presentation of $\tau\Omega\tau^{-1}\Omega^{-1}M$ in the third row of the diagram.  Observe that $\nu\Omega\tau^{-1}\Omega^{-1}M=\nu\text{Im}(\nu^{-1}i_3)=\text{Im}(\nu\nu^{-1}i_3)=\text{Im}\,i_3=I^2_C$.

Let $g_0=\nu\nu^{-1}i_1$.  By commutativity in the diagram below, we see that $i_4\pi_1g_0=i_2=i_4 i_1$.  Since $i_4$ is injective, we conclude $\pi_1 g_0=i_1$.  This shows that we have two short exact sequences as in the statement of the lemma and that the second square in the diagram is commutative.  Then by the universal property of ker$\,\pi_1$ there exists $g_1 \in \text{Hom}_C(M,\tau\Omega\tau^{-1}\Omega^{-1}M)$ that makes the first square commute.  

$$\xymatrix@C-=.3cm @R-=.5cm {&&&\Omega^{-1}M\ar[dr]^{i_4}&&&&&\\
0\ar[r]&M\ar[r]^{i_0}&I^0_C\ar[rr]^{i_2}\ar[ur]^{i_1}&&I^1_C\ar[rr]^{i_3}\ar@{.>}[dd]^{\nu^{-1}}&&I_C^2\ar[r]\ar@{.>}[dd]^{\nu^{-1}}&0&\\
&&&\nu^{-1}I_C^0\ar[dl]_{\nu^{-1}i_1}\ar[dr]^{\nu^{-1}i_2}&&&&\\
&0\ar[r]&\nu^{-1}\Omega^{-1}M\ar[rr]^{\nu^{-1}i_4}\ar@{.>}[ddd]^{\nu}&&\nu^{-1}I^1_C\ar[rr]^{\nu^{-1}i_3}\ar[dr]\ar@{.>}[ddd]^{\nu}&&\nu^{-1}I_C^2\ar[r]&\tau^{-1}\Omega^{-1}M\ar[r]&0\\
&&&&&\Omega\tau^{-1}\Omega^{-1}M\ar[ur]&&\\
&&&I^0_C\ar[dl]_{\nu\nu^{-1}i_1}\ar[dr]^{i_2}&&&&&\\
0\ar[r]&\tau\Omega\tau^{-1}\Omega^{-1}M\ar[r]^{\pi_0}&\nu\nu^{-1}\Omega^{-1}M\ar[rr]^{\nu\nu^{-1}i_4}\ar[dr]_{\pi_1}&&I^1_C\ar[rr]^{i_3}&&I_C^2\ar[r]&0\\
&&&\Omega^{-1}M\ar[ur]_{i_4}&&&&&}$$
\end{proof}

\begin{prop}\label{5.2}
Let {\upshape{gl.dim}}$\,C\leq 2$ and {\upshape{$E={\text{Ext}}^2_C(DC,C)$}}, then for every $M\in$ {\upshape{mod}}$\,C$
{\upshape{$$M\otimes E\cong {\text{Hom}}_C(E,M\otimes E)\otimes E.$$}}
\end{prop}
\begin{proof}
Observe that by Proposition \ref{4.1}(c) the left hand side of the statement above is isomorphic to $\tau^{-1}\Omega^{-1}M$.  Now consider the right hand side 
\begin{align*} \text{Hom}_C(E, M\otimes E)\otimes E  &\cong D(E\otimes D(M\otimes E))\otimes E &\text{by}&\;\text{Lemma \ref{*}} \\  & \cong \tau\Omega(M\otimes E)\otimes E & \text{by}& \;\text{Proposition \ref{4.1}(d)}\\& \cong \tau\Omega\tau^{-1}\Omega^{-1}M\otimes E &\text{by}&\;\text{Proposition \ref{4.1}(c)} \\& \cong \tau^{-1}\Omega^{-1}\tau\Omega\tau^{-1}\Omega^{-1}M &\text{by}&\;\text{Proposition \ref{4.1}(c)}.
\end{align*} 
Therefore, it suffices to show that $\tau^{-1}\Omega^{-1}M\cong \tau^{-1}\Omega^{-1}\tau\Omega\tau^{-1}\Omega^{-1}M$.  If id$_C M\leq 1$ then both sides are zero and the statement follows.   If id$_C M = 2 $, then by Lemma \ref{5.1} we have a short exact sequence 
$$\xymatrix{0\ar[r]&\tau\Omega\tau^{-1}\Omega^{-1}M\ar[r]^-{\pi_0}&\tilde{I}_C\ar[r]&\Omega^{-1}M\ar[r]&0}.$$
Next we construct the following commutative diagram
$$\xymatrix{0\ar[r]&\tau\Omega\tau^{-1}\Omega^{-1}M\ar[r]\ar@{=}[d]& I\ar[r]\ar[d]^{i}&\Omega^{-1}\tau\Omega\tau^{-1}\Omega^{-1}M\ar[r]\ar@{.>}[d]^{j}&0\\
0\ar[r]&\tau\Omega\tau^{-1}\Omega^{-1}M\ar[r]^-{\pi_0}&\tilde{I}_C\ar[r]&\Omega^{-1}M\ar[r]&0}$$
where $I$ is an injective envelope of $\tau\Omega\tau^{-1}\Omega^{-1}M$.   The map $i$ exists and is a monomorphism by the properties of an injective envelope; $j$ exists by the universal property of the cokernel.   Moreover, $\tilde{I}_C\cong i(I)\oplus I'$, where $I'$ is some injective $C$-module.  By commutativity Im$\,\pi_0 \subset i(I)$, which implies that $\Omega^{-1}M\cong I'\oplus \Omega^{-1}\tau\Omega\tau^{-1}\Omega^{-1}M$.  Hence, $\tau^{-1}\Omega^{-1}M\cong \tau^{-1}\Omega^{-1}\tau\Omega\tau^{-1}\Omega^{-1}M$.
\end{proof}

To simplify the notation we will write $I_C$ for an injective $C$-module, and $I_B$ for the corresponding injective $B$-module, that is $I_B=D(B\otimes D I_C)$.  Also, the morphisms in the statements of the proceeding lemmas will be used in the proof of the main theorem, so we keep the notation consistent throughout this section.

\begin{lem}\label{5.3}
Suppose $M$ is a module over a tilted algebra $C$ such that {\upshape{id}}$_C M=2$ and $\textup{pd}_C M\leq 1$.  Let $0\rightarrow M\xrightarrow{i_0} I^0_C$ be an injective envelope of $M$.   Then there is a commutative diagram with exact rows
$$\xymatrix@C=2cm {0\ar[r]&M\ar[r]^-{\bigl(\begin{smallmatrix}
j_0i_0\\ g_1\end{smallmatrix} \bigr)}\ar[d]^{i_0}&I^0_B\oplus \tau\Omega\tau^{-1}\Omega^{-1}M\ar[r]^-{\bigl(\begin{smallmatrix}
\alpha,&-l_1\pi_0
\end{smallmatrix} \bigr)}\ar[d]^{\bigl(\begin{smallmatrix}
1&0\\0&\pi_0
\end{smallmatrix} \bigr)}&\tilde{I}_B\ar[r]^{\eta} \ar@{=}[d]&\textup{Ext}^1_C(E,M)\ar@{=}[d]\\
0\ar[r]&I^0_C\ar[r]^-{\bigl(\begin{smallmatrix}
j_0\\ g_0\end{smallmatrix} \bigr)}&I^0_B\oplus \tilde{I}_C\ar[r]^-{\bigl(\begin{smallmatrix}
\alpha,&-l_1
\end{smallmatrix} \bigr)}&\tilde{I}_B\ar[r]^{\eta}&\textup{Ext}^1_C(E,M)}$$
where $\tilde{I}_C=\nu\nu^{-1}\Omega^{-1}M$ is the injective $C$-module of Lemma \ref{5.1}. 
\end{lem}

\begin{proof}
We begin by constructing the following diagram (\ref{d3}) with commutative squares and exact rows.  The exactness of its top row follows from  Lemma \ref{5.1}.

\begin{equation}\label{d3} \xymatrix@C=2cm{0\ar[r]&M\ar@{=}[d]\ar[r]^-{\bigl(\begin{smallmatrix}
i_0\\ g_1\end{smallmatrix} \bigr)}&I^0_C\oplus \tau\Omega\tau^{-1}\Omega^{-1}M\ar[r]^-{\bigl(\begin{smallmatrix}
g_0,&-\pi_0\end{smallmatrix} \bigr)}\ar[d]^-{\bigl(\begin{smallmatrix}
j_0& 0 \\ 0 & 1\end{smallmatrix} \bigr)}&
\tilde{I}_C\ar[r]\ar[d]^-{l_1}&0\ar[d]\\
0\ar[r]&M\ar[r]^-{\bigl(\begin{smallmatrix}
i_0j_0\\ g_1\end{smallmatrix} \bigr)}&I^0_B\oplus \tau\Omega\tau^{-1}\Omega^{-1}M\ar[r]^-{\bigl(\begin{smallmatrix}
\alpha,&-l_1\pi_0
\end{smallmatrix} \bigr)}&\tilde{I}_B\ar[r]^{\eta}&\text{Ext}^1_C(E,M)}\end{equation}

We apply the coinduction functor Hom$_C(B,-)$ to the top sequence in diagram (\ref{d3}) to obtain the bottom sequence and the corresponding inclusion maps between the two sequences.  First observe that using Proposition \ref{4.1}(d) we have $\tau\Omega\tau^{-1}\Omega^{-1}M \cong D(E\otimes D\tau^{-1}\Omega^{-1} M)$.   Since $C$ is tilted, Lemma \ref{4.4}(b) shows that the projective dimension of $\tau\Omega\tau^{-1}\Omega^{-1}M$ is strictly less than 2, and Proposition \ref{4.3}(b) implies that both $M$ and $\tau\Omega\tau^{-1}\Omega^{-1} M$  do not change under coinduction.  Also, there is a $C$-module isomorphism Ext$^1_C(B,M)\cong \text{Ext}^1_C(C\oplus E,M)=\text{Ext}^1_C(E,M)$.  

Hence, coinducing the top sequence in diagram (\ref{d3}) we obtain the long exact sequence underneath.  Here, $\alpha$ is the coinduced morphism from $g_0$, and $l_1,\,j_0$ are the inclusions as in Proposition~\ref{3.6}(b), such that there is a commutative diagram
 \begin{equation}\label{4}  \xymatrix{0\ar[r]&I^0_C\ar[r]^{j_0}\ar[d]^{g_0}&I^0_B\ar[d]^{\alpha}\\
0\ar[r]&\tilde{I}_C\ar[r]^{l_1}&\tilde{I}_B.} \end{equation}

Moreover, it is easy to see that diagram (\ref{d3}) is commutative.  Indeed, $\eta l_1 = 0$ because $l_1$ factors through the kernel of $\eta$.  Finally, observe that the bottom exact sequence in this diagram coincides with the top exact sequence as in the statement of the lemma.   Therefore, it suffices to show that the bottom row depicted in the lemma is exact and that the diagram is commutative.  

First observe that the second square and the third square in the diagram of the lemma commute trivially while the first one commutes because $\pi_0 g_1=g_0i_0$, by Lemma \ref{5.1}. Next, note that the bottom row is exact at $I^0_C$, because $j_0$ is injective, and we want to show that it is also exact at $\tilde{I}_B$.   That is, we claim that $\text{Im}\,(\alpha, -l_1)=\text{ker}\,\eta$.  Consider $\eta (\alpha, -l_1) = (\eta \alpha, -\eta l_1)$.  The first entry is zero because the top row is exact at $\tilde{I}_B$, and the second entry is zero by commutativity of diagram (\ref{d3}).  Hence, $\text{Im}\,(\alpha, -l_1)\subset \text{ker}\,\eta$.  To show the reverse inclusion, suppose $\eta (e) =0$ for some $e\in \tilde{I}_B$.  Since, the top row is exact at $\tilde{I}_B$ there exist $a\in I^0_B$ and $c\in \tau \Omega \tau^{-1}\Omega^{-1} M$ such that $\alpha(a)-l_1(\pi_0(c))=e$.  This shows that $\text{ker}\,\eta \subset \text{Im}\,(\alpha, -l_1)$, so the bottom sequence in the statement of the lemma is exact at $\tilde{I}_B$.  

Thus it remains to show that Im$\bigl(\begin{smallmatrix} j_0\\g_0\end{smallmatrix}\bigr)$=ker$\bigl(\begin{smallmatrix} \alpha,&-l_1\end{smallmatrix}\bigr)$. To show the forward inclusion consider  $\bigl(\begin{smallmatrix} \alpha,&-l_1\end{smallmatrix}\bigr)$$\bigl(\begin{smallmatrix} j_0\\g_0\end{smallmatrix}\bigr)$=$\alpha j_0-l_1g_0$ which is zero by diagram (\ref{4}). Now suppose $\alpha (a)-l_1(b)=0$ for some $a\in I^0_B$ and $b\in \tilde{I}_C$.  By Lemma \ref{5.1}, there exists $d\in I^0_C$ such that  $i_1(d)=\pi_1(b)$.  Then $\pi_1 g_0 (d)=i_1(d)=\pi_1(b)$ or equivalently $\pi_1(b-g_0(d))=0$. Again Lemma \ref{5.1} implies that there exists $c\in\tau\Omega\tau^{-1}\Omega^{-1} M$ such that $\pi_0(c)=b-g_0(d)$.  Now we have $\alpha(a)=l_1(b)=l_1(\pi_0(c)+g_0(d))=l_1\pi_0(c)+\alpha j_0(d)$, where the last identity follows from diagram (\ref{4}),  so $\alpha(a-j_0(d))-l_1\pi_0(c)=0$.   Using the fact that the top row in our diagram is exact we can find $p\in M$ such that $j_0i_0(p)=a-j_0(d)$ and $g_1(p)=c$.  Finally, $i_0(p)+d\in I^0_C$ and $\bigl(\begin{smallmatrix} j_0\\g_0\end{smallmatrix} \bigr) (i_0(p)+d)=\bigl(\begin{smallmatrix} a\\ b\end{smallmatrix}\bigr)$ since $g_0i_0=\pi_0 g_1$.  This shows the reverse inclusion and completes the proof of the lemma.  
\end{proof}

\begin{lem}\label{5.4}
Let {\upshape{gl.dim}}$\,C = 2$.  Suppose $M$ is a $C$-module with an injective envelope $0\rightarrow M \xrightarrow{i_0} I^0_C$ and {\upshape{id}}$_C M=2$.  Then there exists a commutative diagram with exact rows 
$$\xymatrix@C=1.8cm {0\ar[r]&M\otimes B \ar[r]^-{\bigl(\begin{smallmatrix} u_1\\g_2\end{smallmatrix} \bigr)}\ar@{=}[d]&M \oplus \tau\Omega\tau^{-1}\Omega^{-1} M \otimes B \ar[r]^-{\bigl(\begin{smallmatrix} g_1, & -\delta_1\end{smallmatrix} \bigr)}\ar[d]^{\bigl(\begin{smallmatrix} i_0&0\\0&1\end{smallmatrix}\bigr)}&\tau\Omega\tau^{-1}\Omega^{-1}M \ar[r]\ar[d]^{\pi_0}&0\\
0\ar[r]&M\otimes B \ar[r]^-{\bigl(\begin{smallmatrix} i_0u_1\\g_2\end{smallmatrix} \bigr)}&I^0_C\oplus \tau\Omega\tau^{-1}\Omega^{-1}M\otimes B\ar[r]^-{\bigl(\begin{smallmatrix} g_0,&-\pi_0\delta _1\end{smallmatrix}\bigr)}&\tilde{I}_C\ar[r]&0}
$$
where $\tilde{I}_C=\nu\nu^{-1}\Omega^{-1} M$ is the injective $C$-module of Lemma \ref{5.1}. 
\end{lem}

\begin{proof}
By Lemma \ref{5.1} there is a short exact sequence 
$$\xymatrix@C=2cm{0\ar[r]&M \ar[r]^-{\bigl(\begin{smallmatrix}
i_0\\ g_1\end{smallmatrix} \bigr)}&I^0_C\oplus \tau\Omega\tau^{-1}\Omega^{-1}M \ar[r]^-{\bigl(\begin{smallmatrix}
g_0,&-\pi_0
\end{smallmatrix} \bigr)}&\tilde{I}_C\ar[r]&0.}$$

We will apply the induction functor $-\otimes B$ to this sequence. Since $I^0_C$ and $\tilde{I}_C$ are injective, Proposition \ref{4.3}(a) implies that these modules do not change under induction.  Also, there is a $C$-module isomorphism Ext$^1_C(\tilde{I}_C,DB)\cong \text{Ext}^1_C(\tilde{I}_C, DC\oplus DE)=0$, by Corollary~\ref{4.7}. Therefore, using the fact that  $-\otimes B=D\text{Hom}_C(-,DB)$, we see that applying the induction functor produces the following short exact sequence.
$$\xymatrix@C=2cm {0\ar[r]&M\otimes B \ar[r]^-{\bigl(\begin{smallmatrix} i_0u_1\\g_2\end{smallmatrix} \bigr)}&I^0_C\oplus \tau\Omega\tau^{-1}\Omega^{-1}M \otimes B\ar[r]^-{\bigl(\begin{smallmatrix} g_0,&-\pi_0\delta _1\end{smallmatrix}\bigr)}&\tilde{I}_C\ar[r]&0}$$
Here, $g_2$ is the image of $g_1$ under the induction functor, and $\delta_1,\,u_1$ are projections as in Proposition \ref{3.6}(a), such that there is the following commutative diagram 
\begin{equation}\label{5} \xymatrix{M\otimes B \ar[d]^{g_2}\ar[r]^{u_1}&M \ar[r]\ar[d]^{g_1}&0\\
 \tau\Omega\tau^{-1}\Omega^{-1}M\otimes B \ar[r]^{\delta _1}&\tau\Omega\tau^{-1}\Omega^{-1}M\ar[r]&0.} \end{equation}
Thus we have constructed the bottom row as in the conclusion of the lemma.  It suffices to show that the top row is exact and that the diagram is commutative.  First observe that the first square commutes trivially and the second one commutes because $\pi_0 g_1=g_0i_0$, by Lemma \ref{5.1}.  Next, note that the top row is exact at $M\otimes B$ by commutativity of the diagram and it is exact at $\tau\Omega\tau^{-1}\Omega^{-1} M$ because $\delta_1$ is surjective.  

Therefore, it remains to show that  Im$\bigl(\begin{smallmatrix} u_1\\g_2\end{smallmatrix}\bigr)$=ker$\bigl(\begin{smallmatrix} g_1,&-\delta_1\end{smallmatrix}\bigr)$. The forward inclusion holds, since  $\bigl(\begin{smallmatrix} g_1,&-\delta_1\end{smallmatrix}\bigr)$$\bigl(\begin{smallmatrix} u_1\\g_2\end{smallmatrix}\bigr)$=$g_1u_1-\delta_1g_2=0$, by diagram (\ref{5}).  Now suppose that $g_1(a)-\delta_1(b)=0$ for some $a\in M$ and $b\in \tau\Omega\tau^{-1}\Omega^{-1} M \otimes B$.  Applying $\pi_0$ to both sides we obtain $\pi_0 g_1(a)-\pi_0\delta_1(b)=0$ or equivalently $g_0i_0(a)-\pi_0\delta_1(b)=0$, by commutativity.  Since the bottom row is exact we can find $p\in M\otimes B$ such that $i_0u_1(p)=i_0(a)$ and $g_2(p)=b$.  Note that $i_0$ is injective so $u_1(p)=a$.  Finally, $\bigl(\begin{smallmatrix} u_1\\g_2 \end{smallmatrix}\bigr) (p)$=$\bigl(\begin{smallmatrix} a\\b \end{smallmatrix}\bigr)$. This shows the reverse inclusion and completes the proof of the lemma. 
\end{proof}

The next proposition describes an isomorphism between induction and coinduction for a particular type of modules.
\begin{prop}\label{5.5}
Suppose {\upshape{gl.dim}}$\,C=2$ and $M=\tau\Omega\tau^{-1}\Omega^{-1}N$ for some $N\in$ {\upshape{mod}}$\,C$.  Then the induction of $M$ is isomorphic to the coinduction of $M\otimes E$ and there is a commutative diagram with exact rows 
{\upshape{$$\xymatrix {0\ar[r]&M\otimes E \ar[r]^{1\otimes i}\ar[d]^{\phi}&M\otimes B \ar[r]^{1\otimes \pi}\ar[d]^{\theta}&M\otimes C \ar[r]\ar[d]^{\psi}&0\\
0\ar[r]&\text{Hom}_C(C,M\otimes E)\ar[r]^{\pi^{*}}&\text{Hom}_C(B,M\otimes E)\ar[r]^{i^{*}}&\text{Hom}_C(E,M\otimes E)\ar[r]&0}
$$}}
where $\phi,\theta,$and $\psi$ are isomorphisms of $B$-modules. 
\end{prop}
 
\begin{proof}
First observe that the top row is a short exact sequence by Proposition \ref{3.6}(a), and the bottom row is a short exact sequence by Propositions \ref{3.6}(b) and \ref{3.3}(b).  The maps $i$ and $\pi$ are as in sequence \eqref{(1)}.  Next we explicitly describe the rest of the maps that appear in the diagram above.  

Let us begin with the definition of $\psi$.  First, consider 
\begin{align*}
M&\cong \tau\Omega\tau^{-1}\Omega^{-1}N&&\\&\cong D(E\otimes D\tau^{-1}\Omega^{-1}N)&\text{by}&\;\text{Proposition \ref{4.1}(d)}\\& \cong \text{Hom}_C(E, \tau^{-1}\Omega^{-1} N)&\text{by}&\;\text{Lemma \ref{*} }\\&\cong \text{Hom}_C(E, N\otimes E)&\text{by}&\;\text{Proposition \ref{4.1}(c)}.
\end{align*}   
Then  Proposition \ref{5.2} implies that 
$$\label{9} M\otimes E\cong N\otimes E$$ 
and in turn 
$$\text{Hom}_C(E,M\otimes E)\cong \text{Hom}_C(E,N\otimes E)\cong M.$$ 
Now, there exists a unique $C$-module homomorphism $\psi$ such that 
$$\psi: M\otimes C \rightarrow \text{Hom}_C(E,M\otimes E) $$
$$m\otimes c \longmapsto (e \mapsto mc\otimes e). $$
Next, we want to show that $\psi$ is injective.  Since $M\otimes_C C\cong M$, it suffices to show that $\psi(m\otimes 1)=0$ if and only if $m=0$.  Now suppose that $\psi(m\otimes 1)=0$ for some $m\otimes 1 \in M\otimes C$, which means $m\otimes e=0$, for all $e\in E$, and we need to show that $m=0$.  For this we use the universal property of the tensor product.  Consider the diagram 
$$\xymatrix{M\times E \ar[r] \ar[dr]_{\overline{\psi}}& M\otimes _C E \ar@{.>}[d]^{\Psi}\\& N\otimes E}$$ 
where $\overline{\psi} (m,e)=m(e)$. Recall that we can think of $M$ as Hom$_C (E, N\otimes E)$, so here by $m(e)$ we understand a map $m$ evaluated at an element $e\in E$.  One can check that $\overline{\psi}$ is a $C$-balanced map, and the universal property of the tensor product implies that there exists a unique $C$-module homomorphism $\Psi$ such that $\Psi (m\otimes e)=m(e)$.  Now suppose $m\otimes e =0 $ for all $e \in E$.  Then $\Psi (m\otimes e)=m(e)=0$ for all $e\in E$, which means that $m$ is the zero map, thus $m=0$.    This shows that $\psi$ is an injective $C$-module homomorphism, but since $M\cong \text{Hom}_C (E, M\otimes E)$ are finite dimensional, this shows that $\psi$ is a $C$-module isomorphism.  Because every $C$-module is also a $B$-module by defining the action of $E$ to be trivial, then $\psi$ is also a $B$-module isomorphism.

Now we define $\theta$, and again we use the universal property of the tensor product.  Consider the diagram  
$$\xymatrix{M\times B \ar[r] \ar[dr]_{\varphi}& M\otimes _C B \ar@{.>}[d]^{\theta}\\& \text{Hom}_C(B,M\otimes E)}$$ 
 where $$\varphi : (m, (c,e))\longmapsto ( (c',e')\mapsto m\otimes (ce'+ec')).$$
Again one can easily check that this map is $C$-balanced, so the universal property of the tensor product implies that there exists a unique $C$-module homomorphism $\theta$ such that 
$$\theta (m\otimes (c,e))=((c',e')\mapsto m\otimes(ce'+ec')).$$   
Now we want to show that $\theta$ is a $B$-module homomorphism.  Observe that for all $(\tilde{c},\tilde{e})\in B$ and $m\otimes (c,e)\in M\otimes B$ we have 
\begin{align*}
\theta (m\otimes (c,e))\cdot (\tilde{c},\tilde{e})&=((c',e')\mapsto m\otimes (ce'+ec'))\cdot (\tilde{c},\tilde{e})\\
&=((c',e')\mapsto m\otimes (c(\tilde{c}e'+\tilde{e}c')+e\tilde{c}c')).
\end{align*}
 On the other hand  
\begin{align*}
\theta (m\otimes (c,e)\cdot (\tilde{c},\tilde{e}))&=\theta(m\otimes (c\tilde{c},c\tilde{e}+e\tilde{c}))\\&=((c',e')\mapsto m\otimes (c\tilde{c}e'+(c\tilde{e}+e\tilde{c})c')),
\end{align*}
 and the two expressions are the same.  This shows that $\theta$ is a $B$-module homomorphism.  

Finally, we define the morphism $\phi$.  Let  
$$\phi: M\otimes E \rightarrow \text{Hom}_C (C, M\otimes E)$$ 
$$\phi: m\otimes e \longmapsto (c\mapsto (m\otimes e)\cdot c)$$
which is a standard isomorphism of $C$-modules.  By the same reasoning as above it is also an isomorphism of $B$-modules.  
Thus we defined all morphisms appearing in the proposition, so it remains to show that the corresponding diagram is commutative.  Let $m\otimes e\in M\otimes E$, then consider 
\begin{align*}
\pi^{*}\phi (m\otimes e)&=\pi^{*} (c\mapsto(m\otimes e)\cdot c)\\
&=(c\mapsto (m\otimes e)\cdot c)\circ \pi\\
& =((c',e')\mapsto m\otimes e \cdot \pi (c',e'))\\
&=((c',e')\mapsto m\otimes e \cdot (c',0))\\
&=((c',e')\mapsto m\otimes ec').\end{align*}
 On the other hand 
\begin{align*}
\theta (1\otimes i) (m\otimes e)&=\theta (m\otimes i(e))\\
&=\theta (m\otimes (0,e))\\
&=((c',e')\mapsto m\otimes ec').
\end{align*}
 This shows that the first square commutes.  Now let $m\otimes (c,e)\in M\otimes B$ and consider 
\begin{align*} i^{*}\theta (m\otimes (c,e))&=i^{*}((c',e')\mapsto m\otimes (ce'+ec'))\\&=(e'\mapsto m\otimes ce'). 
\end{align*}
 Also, 
\begin{align*}
\psi(1\otimes \pi) (m\otimes (c,e))&=\psi(m\otimes c)\\&=(e'\mapsto mc\otimes e')\\&=(e'\mapsto m\otimes ce').
\end{align*}
This shows that the second square commutes.  

Next the Five Lemma implies that $\theta$ is a $B$-module isomorphism.  This completes the proof of the proposition.  
\end{proof}

The following corollary is a reformulation of Proposition \ref{5.5} and will be used in the proof of Theorem \ref{5.7}.

\begin{cor}\label{5.6}
Suppose {\upshape{gl.dim}}$\,C=2$ and $M=\tau\Omega\tau^{-1}\Omega^{-1}N$ for some $N\in$ {\upshape{mod}}$\,C$.  Then there is a commutative diagram with exact rows 
{\upshape{$$\xymatrix {0\ar[r]&{M}\otimes E \ar[r]^{\delta_0}\ar[d]^{\phi'}&M\otimes B \ar[r]^{\delta_1}\ar[d]^{\theta'}&M \ar[r]\ar[d]^{\psi'}&0\\
0\ar[r]&N\otimes E\ar[r]^-{\beta_0}&\text{Hom}_C(B,N\otimes E)\ar[r]^-{\beta_1}&M\ar[r]&0}
$$}}
where $\phi ',\theta '$, and $\psi '$ are isomorphisms of $B$-modules. 
\end{cor}
\begin{proof}
First observe that the top row of this diagram is equivalent to the top row of the diagram in Proposition \ref{5.5}.  Also, in the proof of this proposition we showed that $N\otimes E \cong M\otimes E$ and $M\cong \text{Hom}_C(E, M\otimes E)$.  This implies that the bottom rows of the two diagrams are also equivalent.  Thus, we conclude that there exist $\phi', \theta',$ and $\psi'$ that make the diagram above commute.  
\end{proof}

\medskip

We are now ready for our main results. 
The following  theorems \ref{5.7} and \ref{main thm}   give an explicit construction of an injective presentation for each induced module over a cluster-tilted algebra.  

\begin{thm}\label{5.7}
Let $C$ be a tilted algebra, $B$ the corresponding cluster-tilted algebra, $M$ a $C$-module such that $\textup{id}\,_C M \geq 1$ and $\textup{pd}_C M \leq 1$.  Let
 $$\xymatrix {0\ar[r]& M \ar[r]&I^0_C \ar[r]&I^1_C&\text{and}& 0\ar[r]& M\otimes E \ar[r]&\bar{I}^0_C \ar[r]&\bar{I}^1_C}$$
be minimal injective presentations in {\upshape{mod}}$\, C$, and let $\tilde{I}_C$ be the injective $C$-module $\tilde{I}_C=\nu\nu^{-1}\Omega^{-1} M$.  Then 
$$\xymatrix {0\ar[r]&M\otimes B \ar[r]&I^0_B\oplus\bar{I}^0_B \ar[r]&\tilde{I}_B\oplus\bar{I}^1_B\ar[r]&\textup{Ext}^1_C(E,M)}$$
is an exact sequence in {\upshape{mod}}$\,B$ that yields an injective presentation of $M\otimes B$.    
\end{thm}

\begin{remark}
If the injective dimension of $M$ equals one, then $M\otimes E=0$, by Proposition~\ref{4.3}(a), and $\tilde{I}_C=\nu\nu^{-1}\Omega^{-1}M=\nu\nu^{-1}I^1_C=I^1_C$.  Thus in this case the exact sequence becomes  
$$\xymatrix {0\ar[r]& M\otimes B \, = M\ar[r]&I^0_B\ar[r]&I^1_B\ar[r]&\text{Ext}^1_C(E,M).}$$
Moreover, this yields a \emph{minimal} injective presentation of $M\otimes B$ in mod$\,B$. 
\end{remark}

\begin{proof}
If id$_C M = 1$, then  consider a minimal injective resolution of $M$ in mod$\,C$ 
$$\xymatrix{0\ar[r]&M\ar[r]&I^0_C \ar[r]&I^1_C\ar[r]&0}$$
of length at most one.  We apply the coinduction functor Hom$_C(B,-)$ to the injective resolution of $M$ and obtain 
$$\xymatrix{0\ar[r]&M\ar[r]&I^0_B\ar[r]&I^1_B\ar[r]&\text{Ext}_C^1(B,M)}.$$
Indeed the injectives in mod$\,C$ will map to the corresponding injectives in mod$\,B$, and $M$ will not change, since $M\otimes B=M$ by Proposition \ref{4.3}(a).   Also, there is a $C$-module isomorphism Ext$_C^1(B, M)\cong \text{Ext}_C ^ 1 (C\oplus E, M)\cong \text{Ext}_C^1(E,M)$,
and the result follows from the remark.

Now assume that id$_C M =2$.  We start by defining the morphisms in the sequence.  In order to do so, consider the following commutative diagram

\begin{equation} \label{6}
\xymatrix{&0\ar[d]&0\ar[d]&0\ar[d]&\\
0\ar[r]&M\otimes E \ar[r]^-{\beta_0}\ar[d]^{\epsilon}&\text{Hom}_C(B, M\otimes E)\ar[r]^-{\beta_1}\ar[d]^{\gamma_0}&\tau\Omega\tau^{-1}\Omega^{-1}M\ar[r]\ar[d]^{\gamma_1}&0\\
0\ar[r]&\bar{I}^0_C\ar[r]^{v_0}\ar[d]^{\kappa_1}&\bar{I}^0_B\ar[r]^{v_1}\ar[d]^{\pi}&\tau\Omega\bar{I}^0_C\ar[d]^{\gamma_2}\ar[r]&0\\
0\ar[r]&\bar{I}^1_C\ar[r]^{l_2}\ar[d]&\bar{I}^1_B\ar[r]^{\kappa_2}\ar[d]&\tau\Omega\bar{I}^1_C\ar[d]\ar[r]&0\\
&0&0&0&}
\end{equation}
where every row and column is exact.  To construct this diagram, we begin with the injective resolution in mod$\,C$ of $M\otimes E$ as in the statement of the theorem. Note that, by Lemma~\ref{4.4}(a), id$_C M\otimes E \leq 1$.  This sequence appears in the left most column of the diagram above.  Then we apply the coinduction functor Hom$_C(B,-)$ to this sequence and recall that Ext$_C^1(E, M\otimes E)=0$, by Corollary \ref{4.6}(c).  This gives us the short exact sequence to the right of the one we started with, which is the middle column of the diagram.  We also obtain inclusions $\beta_0, v_0,l_2$ from the given $C$-modules to the corresponding coinduced $B$-modules and projections $\beta_1, v_1, k_2$ to the corresponding cokernels as in Proposition \ref{3.6}(b).  By the same proposition, cok$\, \beta_0 = D(E\otimes D(M\otimes E))\cong \tau\Omega\tau^{-1}\Omega^{-1}M$, where the last identity follows from Proposition \ref{4.1}.  Similarly, cok$\,v_0\cong \tau\Omega \bar{I}_C^0$, and cok$\, l_2 \cong \tau \Omega\bar{I}_C ^1$.  Thus we obtain the commutative diagram (\ref{6}).

Now we construct the commutative diagram (\ref{8}) with exact rows, which appears below.   We obtain the bottom two rows from Corollary \ref{5.6} by letting $N=M$. Next, we draw a commutative diagram (\ref{5}) in the top right corner, where the maps $g_1$ and $g_2$ are as in Lemma~\ref{5.4}.   We complete the top row by $M\otimes E$, the kernel of $u_1$, as in Proposition \ref{3.6}(a).  Finally, by the universal property of ker$\,\delta_1$ there exists a morphism $\epsilon_0$ that completes the diagram and makes the upper left square commute.  Note that the bottom row of diagram (\ref{8}) is the same as the top row of diagram (\ref{6}).  

\begin{equation} \label{8}
\xymatrix{ 0\ar[r]& M\otimes E \ar[r]^{u_0}\ar[d]^{\epsilon_0}& M\otimes B\ar[r]^{u_1}\ar[d]^{g_2}&M \ar[r]\ar[d]^{g_1}&0\\
0\ar[r]&M\otimes E\ar[r]^-{\delta_0}\ar[d]^{\phi'}&\tau\Omega\tau^{-1}\Omega^{-1}M\otimes B\ar[r]^-{\delta_1}\ar[d]^{\theta'}&\tau\Omega\tau^{-1}\Omega^{-1}M\ar[r]\ar[d]^{\psi'}&0\\
0\ar[r]&M\otimes E \ar[r]^-{\beta_0}&\text{Hom}_C(B, M\otimes E)\ar[r]^-{\beta_1}&\tau\Omega\tau^{-1}\Omega^{-1}M\ar[r]&0.}
\end{equation}

Next we want to show that $\epsilon_0$ is an isomorphism.  Observe that, since $M\otimes E$ is finite dimensional, it suffices to show that $\epsilon_0$ is injective.  Suppose $\epsilon_0 (a)=0$, for some $a\in M\otimes E$.  Let $u_0(a)=b$.  Because $u_0$ is injective, it is enough to show that $b=0$.   By commutativity $g_2(b)=0$, but looking at the top row of Lemma \ref{5.4}, we conclude that either $b=0$ or $u_1(b)\not=0$. In the first case we are done, so suppose $u_1(b)\not=0$.  Then $u_1(b)=u_1u_0(a)\not=0$, which is a contradiction since the top row of the diagram (\ref{8}) is exact.  This shows that $a=b=0$ and that $\epsilon_0$ is an isomorphism.  

Now we show that the following sequence as in the statement of the theorem is exact 
\begin{equation}\label{eq 7.5} 
\xymatrix@C=2cm{0\ar[r]&M\otimes B \ar[r]^-{\bigl(\begin{smallmatrix} j_0i_0u_1\\ \gamma_0\theta' g_2\end{smallmatrix}\bigr)}&I^0_B\oplus \bar{I}^0_B\ar[r]^-{\bigl( \begin{smallmatrix} \alpha& -l_1\gamma v_1\\ 0& \pi \end{smallmatrix}\bigr)} & \tilde{I}_B\oplus \bar{I}^1_B\ar[r]^-{\bigl( \begin{smallmatrix} \eta & 0 \end{smallmatrix}\bigr)}&\text{Ext}^1_C(E,M)}
\end{equation}
where the maps are $v_1, \pi, \gamma_0$ are given in diagram (\ref{6}), the maps $\theta', u_1, g_2$ in diagram (\ref{8}), the maps $i_0, j_0, \alpha, l_1, \eta$ in Lemma \ref{5.3}, and $\gamma$ in diagram (\ref{7}) below.   
\begin{equation}\label{7} \xymatrix{\tau\Omega\tau^{-1}\Omega^{-1}P_C\ar[r]^-{\pi_0}\ar[d]_{\gamma_1 \psi'}&\tilde{I}_C\\ \tau\Omega\bar{I}^0_C\ar@{.>}[ur]_{\gamma}&}
\end{equation} 
Here $\pi_0$ is the map of Lemma \ref{5.1}. Observe that $\gamma$ exists because $\tilde{I}_C$ is injective and $\gamma_1 \psi'$ is an injective map.  Moreover, the map $\gamma$ makes the diagram commute, that is $\gamma\gamma_1\psi'=\pi_0$.

First we show that the sequence we defined above is exact at $M\otimes B$. Suppose $\bigl(\begin{smallmatrix}j_0i_0u_1\\\gamma_0\theta' g_2\end{smallmatrix}\bigr)(p)=0$ for some $p\in M\otimes B$.  So on the one hand $j_0i_0u_1(p)=0$, but $j_0$ and $i_0$ are injective, which means $u_1(p)=0$.  By diagram (\ref{8}) there exists $b\in M\otimes E$ such that $u_0(b)=p$.  On the other hand, $\gamma_0\theta' g_2 u_0 (b)=0$, but by commutativity in diagram (\ref{8}) this is equivalent to $\gamma_0\beta_0 \phi' \epsilon_0 (b)=0$.  Because, all of these maps are injective it follows that $b=0$, which implies $p=0$.  This shows that $\bigl(\begin{smallmatrix} j_0i_0u_1\\ \gamma_0\theta' g_2\end{smallmatrix}\bigr)$ is an injective map. 

Next we show that im$\bigl(\begin{smallmatrix} j_0i_0u_1\\ \gamma_0\theta' g_2\end{smallmatrix}\bigr)$=ker$\bigl( \begin{smallmatrix} \alpha& -l_1\gamma v_1\\ 0& \pi \end{smallmatrix}\bigr)$, meaning  that the sequence is exact at $I^0_B\oplus \bar{I}^0_B$.  To show the forward inclusion consider $\bigl( \begin{smallmatrix} \alpha& -l_1\gamma v_1\\ 0& \pi \end{smallmatrix}\bigr) \bigl(\begin{smallmatrix} j_0i_0u_1\\ \gamma_0\theta' g_2\end{smallmatrix}\bigr)=\bigl(\begin{smallmatrix}\alpha j_0 i_0 u_1-l_1\gamma v_1 \gamma_0 \theta' g_2\\ \pi\gamma_0\theta' g_2 \end{smallmatrix}\bigr)$.  Observe that $\pi\gamma_0=0$ by diagram (\ref{6}), which means that the bottom entry is zero.  The top entry is also zero, because 
\begin{align*}
\alpha j_0 i_0 u_1-l_1\gamma v_1 \gamma_0 \theta' g_2&=l_1(\pi_0g_1u_1-\gamma v_1 \gamma_0 \theta' g_2)&\text{by} & \;\text{first row in Lemma \ref{5.3}}\\&=l_1(\gamma\gamma_1\psi' g_1 u_1-\gamma v_1 \gamma_0 \theta' g_2)&\text{by}&\;\text{diagram (\ref{7})}\\&=l_1\gamma (\gamma_1 \psi' g_1 u_1-v_1\gamma_0 \theta' g_2)&&\\& = l_1\gamma (\gamma_1\beta_1 \theta' g_2-v_1\gamma_0\theta' g_2)&\text{by}&\;\text{ diagram (\ref{8})},
\end{align*}
which is zero by commutativity of diagram (\ref{6}).  To show the reverse inclusion suppose $\bigl( \begin{smallmatrix} \alpha& -l_1\gamma v_1\\ 0& \pi \end{smallmatrix}\bigr) \bigl( \begin{smallmatrix} a\\ b\end{smallmatrix}\bigr) =\bigl( \begin{smallmatrix} 0\\ 0 \end{smallmatrix}\bigr)$ for some $a\in I^0_B, b\in \bar{I}^0_B$.  Since $\alpha (a) -l_1\gamma v_1 (b)=0$ and the bottom row in the diagram of Lemma \ref{5.3} is exact, there exists $c\in I^0_C$ such that $j_0(c)=a$ and $g_0(c)=\gamma v_1(b)$.  Also, we have $\pi(b)=0$, so by diagram (\ref{6}) there exists $d\in \text{Hom}_C(B, M\otimes E)$, such that $\gamma_0 (d)=b$.  Since $\theta'$ is an isomorphism we can find $e\in \tau\Omega\tau^{-1}\Omega^{-1} P_C \otimes B$ such that $\theta'(e)=d$.  Now we have 
\begin{align*}
g_0(c)&=\gamma v_1 \gamma _0 \theta' (e) \\&= \gamma\gamma_1\beta_1\theta'(e)&\text{by}&\;\text{diagram (\ref{6})}\\&=\gamma\gamma_1\psi'\delta_1 (e)&\text{by}&\;\text{diagram (\ref{8})}\\&=\pi_0\delta_1(e)&\text{by}&\;\text{diagram (\ref{7}).}
\end{align*}
Equivalently we can write $g_0(c)-\pi_0 \delta _ 1 (e) = 0$, and since the bottom row in the diagram of Lemma \ref{5.4} is exact, there exists $p\in M\otimes B$ such that $i_0 u_i(p)=c$ and $g_2(p)=e$.  Now consider $\bigl( \begin{smallmatrix} j_0 i_0 u_1 \\ \gamma _0 \theta' g_2 \end{smallmatrix}\bigr) (p)= \bigl( \begin{smallmatrix} j_0 (c)\\ \gamma_0 (d)   \end{smallmatrix}\bigr)=\bigl( \begin{smallmatrix} a\\ b \end{smallmatrix}\bigr)$.  This shows exactness at $I^0_B \oplus \bar{I}^0_B$.  

It remains to show that the sequence is exact at $\tilde{I}_B \oplus \bar{I}_B$.  In other words, we claim that $\text{Im}\bigl( \begin{smallmatrix} \alpha& -l_1\gamma v_1 \\ 0 & \pi\end{smallmatrix}\bigr)=\text{ker}\bigl( \begin{smallmatrix} \eta & 0\end{smallmatrix}\bigr)$.  Consider $\bigl( \begin{smallmatrix} \eta & 0\end{smallmatrix}\bigr)\bigl( \begin{smallmatrix} \alpha& -l_1\gamma v_1 \\ 0 & \pi\end{smallmatrix}\bigr)=\bigl( \begin{smallmatrix} \eta\alpha & -\eta l_1 \gamma v_1\end{smallmatrix}\bigr)$.  By commutativity and exactness of diagram (\ref{d3}) it follows that $\eta \alpha =0$ and $\eta l_1 =0$.  Therefore, $\text{Im}\bigl( \begin{smallmatrix} \alpha& -l_1\gamma v_1 \\ 0 & \pi\end{smallmatrix}\bigr)\subset \text{ker}\bigl( \begin{smallmatrix} \eta & 0\end{smallmatrix}\bigr)$.

To show the reverse inclusion, suppose $\bigl( \begin{smallmatrix} \eta & 0\end{smallmatrix}\bigr) \bigl( \begin{smallmatrix} a \\ b\end{smallmatrix}\bigr)=0$ for some $a\in \tilde{I}_B$, $b\in\bar{I}^1_B$.  Since $\pi$ is surjective it suffices to show that if $\eta(a)=0$ then $a\in \text{Im} \, \bigl( \begin{smallmatrix} \alpha & -l_1 \gamma  v_1 \end{smallmatrix}\bigr)$.  That is, given $a\in \tilde{I}_B$ find $c\in I^0_B,\; d\in \bar{I}^0_B$ such that $\alpha(c)-l_1\gamma v_1 (d) =a$.  Since $\eta(a)=0$, Lemma \ref{5.3} implies that there exist $c'\in I^0_B$ and $e\in \tilde{I}_C$ such that $\alpha (c')-l_1(e)=a$.  Then by Lemma \ref{5.4},  there exist $m\in I^0_C$ and $n\in \tau\Omega\tau^{-1}\Omega^{-1} M \otimes B$ such that $g_0(m)-\pi_0 \delta_1 (n)=e$.  Finally, let $d=\gamma _0 \theta' (n) \in \bar{I}^0_B$ and $c=j_0(m)+c' \in I^0_B$ and observe that  
\begin{align*}
\alpha (j_0(m)+c')-l_1\gamma v_1 (\gamma_0 \theta' (n))&=l_1g_0(m)+\alpha (c')-l_1\gamma v_1 \gamma_0 \theta' (n)&\text{by}&\;\text{diagram (\ref{4})}\\&=\alpha(c')+l_1(g_0 (m)-\gamma v_1 \gamma_0\theta' (n))&&\\&=\alpha(c')-l_1(g_0(m)-\gamma\gamma_1\psi'\delta_1(n))&\text{by}&\;\text{diagrams (\ref{6}), (\ref{8})}\\&=\alpha(c')-l_1(g_0(m)-\pi_0\delta_1(n))&\text{by}&\;\text{diagram (\ref{7})}\\&=a.
\end{align*}
This shows the claim and finishes the proof of the theorem.
\end{proof}

\begin{remark}
The theorem above can be dualized.  That is, in a similar manner one can construct projective presentations of coinduced $B$-modules.
\end{remark}

As a corollary, we obtain an injective resolution for each projective $B$-module.
\begin{cor}\label{cor proj} With the notation of Theorem \ref{5.7}, 
 if $M= P_C$ is a projective $C$-module, then  
 $$\xymatrix {0\ar[r]& P_C\otimes B \ar[r]&I^0_B\oplus\bar{I}^0_B \ar[r]&\tilde{I}_B\oplus\bar{I}^1_B\ar[r]&0}$$
 is an injective resolution in $\textup{mod}\,B$.
\end{cor}

\begin{proof}
 By the Theorem \ref{5.7} we obtain an injective presentation of $ P_C\otimes B$ in mod$\,B$.  However, by Corollary \ref{4.7} this presentation becomes an injective resolution since $\text{Ext}_C^1(E,P_C)=0$.
\end{proof}

We obtain therefore a new proof of the following result which was first proved by Keller and Reiten using cluster categories \cite{KR}.
\begin{cor}
If $C$ is a tilted algebra and $B$ is the corresponding cluster-tilted algebra, then $B$ is 1-Gorenstein. 
\end{cor}

\begin{proof}  Suppose $P_B$ is a projective $B$-module.  Proposition \ref{3.4} implies that for any tilted algebra $C$ such that $B=C\ltimes E$ there exists $P_C$, a projective $C$-module, such that $P_B = P_C \otimes B$. 
Corollary \ref{cor proj} implies that the injective dimension of $P_B$ in mod$\,B$ is at most one.  Dually, one can also show that the projective dimension of injective $B$-modules is at most one.  This implies that $B$ is 1-Gorenstein.  
\end{proof}

\medskip
The only modules $M$ for which Theorem \ref{5.7} does not apply are the injective $C$-modules and the $C$-modules of projective dimension 2. We have a similar result for these modules, but we need to consider the indecomposable summands separately as follows.

\begin{thm}\label{main thm}
Let $C$ be a tilted algebra, $B$ the corresponding cluster-tilted algebra, $M$ an \emph{indecomposable} $C$-module. Let 
$$\xymatrix {0\ar[r]& M \ar[r]&I^0_C \ar[r]&I^1_C&\text{and}& 0\ar[r]& M\otimes E \ar[r]&\bar{I}^0_C \ar[r]&\bar{I}^1_C}$$
be minimal injective presentations in {\upshape{mod}}$\, C$, 
$$\xymatrix{0\ar[r] & 
{\tau\Omega M}\ar[r] & \hat{I}_C}$$
an injective envelope in \upshape{mod}$\,C$, and let $\tilde{I}_C$ be the injective $C$-module $\tilde{I}_C=\nu\nu^{-1}\Omega^{-1} M$.  Then 
$$\xymatrix {0\ar[r]&M\otimes B \ar[r]&I^0_B\oplus\bar{I}^0_B \ar[r]&\tilde{I}_B\oplus\bar{I}^1_B\oplus \hat{I}_B}$$
is an injective presentation of $M\otimes B$ in {\upshape{mod}}$\,B$. 
\end{thm}

\begin{proof}
We begin by defining the maps in the injective presentation of $M\otimes B$, so let 

$$\xymatrix@C=2cm{0\ar[r]&M\otimes B \ar[r]^-{\bigl(\begin{smallmatrix} j_0i_0u_1\\ \gamma_0\theta' g_2\end{smallmatrix}\bigr)}&I^0_B\oplus\bar{I}^0_B \ar[r]^-{\Bigl( \begin{smallmatrix} \alpha& -l_1\gamma v_1\\ 0& \pi\\ \zeta& 0 \end{smallmatrix}\Bigr)}&\tilde{I}_B\oplus\bar{I}^1_B\oplus \hat{I}_B}$$
where all the maps except $\zeta$ are the same as in the sequence (\ref{eq 7.5}).\footnote{Note that in (\ref{eq 7.5}) we referred to Lemma \ref{5.1} to define the maps $i_0, j_0, \alpha, l_1$, and, although this lemma assumes $\textup{id}_CM=2$, this condition is not necessary to define these maps.} 
%
%
Therefore, it remains to define the map $\zeta$, which we do next.  

Observe that if $
{\tau\Omega M}=0$ then $\hat{I}_B=0$ and the map $\zeta=0$.   Now, suppose 
{$\tau\Omega M$ is nonzero. 
Recall that  $\tau\Omega M=D(E\otimes DM)$ by Proposition~\ref{4.1}(d)}, and hence Proposition~\ref{4.3}(b) implies that pd$_C M$=2.  Since, $M$ is indecomposable and the algebra $C$ is tilted, it follows that id$_C M \leq 1$.   Therefore, the injective presentation of $M$ in mod$\,C$ given in the statement of the theorem is actually an injective resolution of $M$, see the top row in diagram~(\ref{d10}) below.  
{The bottom row of that diagram is obtained by applying
the coinduction functor to this resolution. We}
 obtain the following commutative diagram where $i^*_0$ and $i_1^*$ are coinduced morphism from $i_0$ and $i_1$ respectively, and the maps $\pi^*, j_0$ are the inclusions as in Proposition~\ref{3.6}(b). 

\begin{equation}\label{d10}
\xymatrix@C=2cm{0\ar[r]&M\ar[r]^-{i_0}\ar[d]^-{\pi^*}&I^0_C\ar[d]^-{j_0}\ar[r]^-{i_1}&I^1_C\ar[d]\ar[r]&0\\
0\ar[r]&D(B\otimes DM)\ar[r]^-{i^*_0}&I^0_B\ar[r]^-{i^*_1}&I^1_B}
\end{equation}

Now we construct a commutative diagram~(\ref{d9}) with exact rows and columns.  By diagram~(\ref{d10}) the first square in diagram~(\ref{d9}) commutes.  Observe, that the vertical sequence on the left in diagram~(\ref{d9}) is a short exact sequence by Proposition~\ref{3.6}(b).  Since $M\xrightarrow{i_0} I^0_C$ is an injective envelope of $M$ in mod$\,C$, the corresponding coinduced map $i^*_0$ is an injective envelope of $D(B\otimes DM)$ in mod$\,B$.  Therefore, the cokernel of $i^*_0$ is $\Omega^{-1}_B(D(B\otimes DM))$. This shows that the bottom sequence in diagram~(\ref{d9}) is indeed a short exact sequence.  Also, since soc$\,I^0_C$ = soc$\,I^0_B$, it follows that $j_0 i_0$ is an injective envelope of $M$ in mod$\,B$, so the cokernel of $j_0 i_0$ is $\Omega^{-1}_B(M)$.  This shows that the top sequence in diagram~(\ref{d9}) is a short exact sequence.  Because $h^*_0j_0i_0=0$, it follows by the universal property of cok$\,j_0i_0$ that there exists a map $h$ as in the diagram, such that the second square commutes. 
{In particular,} $h$ must be surjective, and by the Snake Lemma it follows that ker$\,h\cong D(E\otimes DM)$.  This shows that the vertical sequence on the right in diagram~(\ref{d9}) is also exact.  Therefore, this diagram is commutative with exact rows and columns.  

\begin{equation}\label{d9}
\xymatrix{&&&0\ar[d]\\
&0\ar[d]&&D(E\otimes DM)\ar[d]^-{i^*}\\
0\ar[r]&M\ar[r]^-{j_0i_0}\ar[d]^-{\pi^{*}}&I^0_B\ar[r]^-{h_0}\ar@{=}[d]&\Omega^{-1}_B(M)\ar[r]\ar[d]^-{h}&0\\
0\ar[r]&D(B\otimes DM)\ar[d] \ar[r]^-{i^{*}_0} & I^0_B \ar[r]^-{h^*_0} & \Omega^{-1}_B(D(B\otimes DM))\ar[d]\ar[r]&0\\
&D(E\otimes DM)\ar[d]&&0\\
&0}
\end{equation}

Next we construct an injective  envelope of $\Omega^{-1}_B(M)$. 
{To do so we will apply the Horseshoe Lemma to the short exact sequence in the right column of diagram~(\ref{d9}).}
 Since $D(E\otimes DM)\xrightarrow{f_0}\hat{I}_C$ is an injective {envelope} 
 in mod$\,C$, it follows that the corresponding coinduced map $D(E\otimes DM)\xrightarrow{f^*_0}\hat{I}_B$ is an injective {envelope} 
 in mod$\,B$.  Also, 
 {since
  Im$\,h^*_0 = \text{Im}\,i_1^*$, we see that there is  an injective envelope $\Omega^{-1}(D(B\otimes DM))\xrightarrow{q_0}I^1_B$, 
   where 
\begin{equation}\label{eq 10.5}q_0 h^*_0 = i_1^*.\end{equation} 
   Thus the Horseshoe Lemma implies that there is an exact sequence}

\begin{equation}\label{d9.5}\xymatrix{0\ar[r]&\Omega^{-1}_B(M)\ar[r]^-{\bigl(\begin{smallmatrix} q_0h \\ \zeta_0 \end{smallmatrix}\bigr)}&I^1_B\oplus \hat{I}_B},\end{equation}
 where $\zeta_0 i^* = f^*_0$ and $\zeta_0$ exists because $i^*$ is a monomorphism and $\hat{I}_B$ is an injective $B$-module.  Finally, we define the map $\zeta = \zeta_0 h_0$.  Hence, we defined all morphisms that appear in the injective presentation constructed in this theorem, and now we want to show that this sequence is exact.

\medskip
If id$_C M\geq 1$ and pd$_C M \leq 1$, then by Proposition~\ref{4.3}(b) we see that $\tau\Omega M = D(E\otimes DM)=0$.   This shows that $\hat{I}_C=\hat{I}_B=0$, and the injective presentation in the theorem coincides with the injective presentation in Theorem~\ref{5.7}.  Hence in this case we are done. 

If id$_C M = 0$ or pd$_C M =2$, then id$_C M \leq 1$, because $M$ is indecomposable and $C$ is a tilted algebra.  Proposition~\ref{4.3}(a) implies that $M\otimes B \cong  M$, so in particular $M\otimes E = 0$.  Therefore, $\bar{I}^i_C=\bar{I}^i_B=0$ for $i=0,1$.  Also, since id$_C M \leq 1$ we see that $\Omega^{-1}M = I^1_C$, and we have $\tilde{I}_C = \nu\nu^{-1}\Omega^{-1}M = \nu\nu^{-1} I^1_C = I^1_C$.  Hence, to prove the theorem it suffices to show that 

\begin{equation}\label{s14}
\xymatrix@C=2cm {0\ar[r]&M \ar[r]^-{j_0i_0}&I^0_B \ar[r]^-{\bigl( \begin{smallmatrix} \alpha \\ \zeta \end{smallmatrix}\bigr)}&{I}^1_B\oplus \hat{I}_B}
\end{equation}
is an injective presentation in mod$\,B$, where the maps $j_0, i_0, \alpha$ are given in Lemma \ref{5.3}, and the map $\zeta$ was defined earlier in the proof.  Observe that since id$_C M \leq 1$, we have $\alpha = i_1^*$.  

Because of diagram (\ref{d9}) and the exact sequence (\ref{d9.5}) it suffices to show that
 $\bigl( \begin{smallmatrix} i_1^* \\ \zeta \end{smallmatrix}\bigr) = \bigl( \begin{smallmatrix} q_0 h \\ \zeta_0 \end{smallmatrix}\bigr) h_0$. By definition we have that $\zeta=\zeta_0 h_0$.
 On the other hand, 
$q_0 h h_0 =  q_0 h_0^* =i_1^*$, where the first equality holds  by diagram~(\ref{d9}) and the second by equation (\ref{eq 10.5}).  This shows the claim, and completes the proof of the theorem.   
\end{proof}

{The injective presentation of $M\otimes B$ described in Theorem \ref{main thm} can be simplified depending on the homological dimensions of $M$.  Recall, that if id$_C M = 1$ and pd$_C M \leq 1$ then an injective presentation of $M\otimes B = M$ is given in the remark following Theorem~\ref{5.7}.  However, if id$_C M = 2$ and pd$_C M \leq 1$, then the construction in Theorem~\ref{5.7} is needed to produce an injective presentation.  Finally, if id$_C M = 1$ and pd$_C M =2$ then $M\otimes B = M$ and its injective presentation is given in (\ref{s14}).  We examine the remaining cases next.  }

\begin{cor}\label{cor 11}
With the notation of Theorem~\ref{main thm}, if $M=I_C$ is an injective $C$-module, then 
$$\xymatrix{0\ar[r]& I_C\otimes B = I_C \ar[r]& I_B \ar[r] & \hat{I}_B}$$
is an injective presentation in \textup{mod}$\,B$.   
\end{cor}

\begin{proof}
Consider the injective presentation in Theorem~\ref{main thm} with $M=I_C$.  Since, $I_C$ is injective it follows that $I^0_C = I_C$ and $I^1_C = 0$.  Therefore, $I^0_B = I_B$ and $\tilde{I}_C = \tilde{I}_B = 0$.  On the other hand, Proposition~\ref{4.3}(a) implies that $M\otimes E =0$, so $\bar{I}^i_C = \bar{I}^i_B = 0$ for $i=0,1$.  This shows that when $M=I_C$ then the injective presentation in Theorem~\ref{main thm} coincides with the exact sequence given in this corollary.  
\end{proof}

\begin{remark}
If $M = I_C$ is an injective $C$-module such that pd$_C I_C \leq 1$, then the injective presentation in Corollary~\ref{cor 11} becomes an injective resolution with $\hat{I}_B=0$.  Indeed, Proposition~\ref{4.1}(d) implies that $\tau \Omega I_C = D(E\otimes D I_C)$, which is zero by Proposition~\ref{4.3}(b).  Therefore, its injective cover $\hat{I}_C = \hat{I}_B = 0$.  In particular, observe that in this case $I_C = I_B$ is an injective $B$-module.   
\end{remark}

\section{Examples}\label{sect ex}
In this section, we provide examples of injective resolutions of induced modules that are constructed by applying Theorem~\ref{main thm} and its corollaries.  In certain cases, we can also obtain infinite periodic injective resolutions of such modules, see Example~\ref{ex 3}.

\begin{exmp}
Let $C$ be the tilted algebra  of type $\mathbb{D}_5$ given by the following quiver with relations.
$$\xymatrix@R=.5cm{1\ar[dr]&&&\\&3\ar[r]^{\alpha}\ar[dl]&4\ar[r]^{\beta}&5&&\alpha\beta=0.\\2&&&}$$
The corresponding cluster-tilted algebra $B$ is given by the quiver with relations below. 
$$\xymatrix@R=.5cm{1\ar[dr]&&&\\&3\ar[r]^{\alpha}\ar[dl]&4\ar[r]^{\beta}&5\ar@/^1pc/[ll]^{\delta}&&\alpha\beta=\beta\delta=\delta\alpha=0.\\2&&&}$$
We want to construct an injective presentation of $M =\begin{smallmatrix}3\\2 \end{smallmatrix}$ in mod$\,B$.  Observe that $M$ is also an indecomposable $C$-module of projective dimension 2, so its injective presentation is given in (\ref{s14}).  The injective resolution of $M$ is mod$\,C$ is given below. 
$$\xymatrix{0\ar[r]&M \ar[r]&I_C(2)\ar[r]&I_C(1)\ar[r]&0}$$
Next, we see that $\tau_C \Omega_C M = 5$, hence $\hat{I}_C = {I}_C(5)$ is the injective envelope of $5$.  Therefore, we obtain the following injective presentation of $M = M\otimes _C B$ in mod$\,B$.  
$$\xymatrix{0\ar[r]&M \ar[r]&I_B(2)\ar[r]&I_B(1)\oplus I_B(5)}$$
Note that here $I_B(2)\not = I_C(2)$.
\end{exmp}

\begin{exmp}
Let $C$ be the tilted algebra given by the following quiver with relations. 
$$
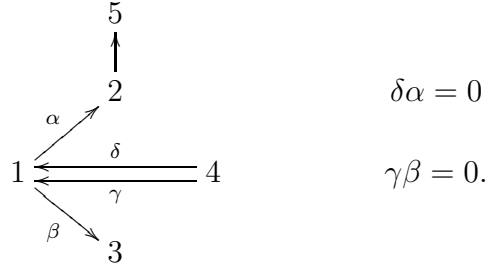
{\xymatrix@R=.5cm {&5&\\&2\ar[u]&&&\delta\alpha=0\\1\ar[ur]^{\alpha}\ar[dr]_{\beta}&&4 \ar@<-.5ex>[ll]_{\delta} \ar@<.5ex>[ll]^{\gamma}&&\gamma\beta = 0.\\&3&}}$$
The corresponding cluster-tilted algebra $B$ is of type $\tilde{\mathbb{A}}_{(3,2)}$ and it is given by the quiver with relations below. 
$$
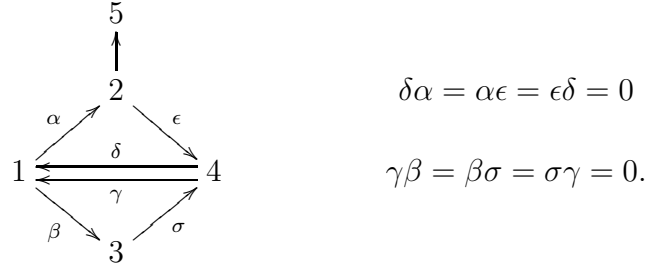
{\xymatrix@R=.5cm {&5&\\&2\ar[dr]^{\epsilon}\ar[u]&&&\delta\alpha=\alpha\epsilon=\epsilon\delta=0\\1\ar[ur]^{\alpha}\ar[dr]_{\beta}&&4 \ar@<-.5ex>[ll]_{\delta} \ar@<.5ex>[ll]^{\gamma}&&\gamma\beta = \beta\sigma=\sigma\gamma=0.\\&3\ar[ur]_{\sigma}&}}$$
We want to construct the injective resolution of $P_B(2)=\begin{smallmatrix}2\\5\;4\\\;\;\;1\\\;\;\;2\\\;\;\;5\end{smallmatrix}$, the projective $B$-module at vertex 2, which is given in Corollary \ref{cor proj}.  First, we find minimal injective presentations of $P_C(2)=\begin{smallmatrix}2\\5\end{smallmatrix}$ and $P_C(2)\otimes E = \begin{smallmatrix}4\\1\\2\\5\end{smallmatrix}$ in mod$\,C$, which are given below.  
$$\xymatrix{0\ar[r]&P_C(2)\ar[r]&I_C(5)\ar[r]&I_C (1)&&0\ar[r]&P_C(2)\otimes E\ar[r]&I_C(5)\ar[r]&0}$$
We write out the explicit representations involved in the sequences above. 
$$\xymatrix @C=1.5cm {0\ar[r]&{\begin{smallmatrix}2\\5\end{smallmatrix}}\ar[r]&{\begin{smallmatrix}4\\1\\2\\5\end{smallmatrix}}\ar[r]&{\begin{smallmatrix}4\;4\\1\end{smallmatrix}}&
0\ar[r]&{\begin{smallmatrix}4\\1\\2\\5\end{smallmatrix}}\ar[r]&{\begin{smallmatrix}4\\1\\2\\5\end{smallmatrix}}\ar[r]&0}$$
Here $I_C(i)$ denotes the injective $C$-module at vertex $i$, while $I_B(i)$ will denote the corresponding injective $B$-module.  Next, we calculate $\nu\nu^{-1}\Omega ^{-1} P_C(2)$.  Observe that $\Omega^{-1} P_C (2)$ is the module with dimension vector $(1,0,0,1,0)$ such that $\gamma = 1$ and $\delta=0$.  Then $\nu^{-1}\Omega^{-1}P_C(2)=\text{Hom}_C(DC, \Omega^{-1}P_C(2))$, and the only injectives that have a nonzero morphism into $\Omega^{-1} P_C(2)$ are $I_C(2)$ and $I_C(5)$.  Hence, $\nu^{-1}\Omega^{-1} P_C(2)=P_C(2)$, so $\nu\nu^{-1}\Omega^{-1}P_C(2)=I_C(2)$. Then according to Theorem \ref{5.7} we construct the injective resolution of $P_B(2)$ 
$$\xymatrix{0\ar[r]&P_B(2)\ar[r]&I_B(5)\oplus I_B(5)\ar[r]&I_B(2)\oplus 0\ar[r]&0}$$
or equivalently substituting the representations we have 
$$\xymatrix@C=1.5cm{0\ar[r]&{\begin{smallmatrix}2\\5\;4\\\;\;\;1\\\;\;\;2\\\;\;\;5\end{smallmatrix}} \ar[r]&{\begin{smallmatrix}2\\4\\1\\2\\5\end{smallmatrix}}\oplus {\begin{smallmatrix}2\\4\\1\\2\\5\end{smallmatrix}}\ar[r]&{\begin{smallmatrix}2\\4\\1\\2\end{smallmatrix}}\ar[r]&0.}$$
\end{exmp}

\begin{exmp}\label{ex 3}
 
Consider the following three tilted algebras $C_1,C_2,C_3$ of type $\mathbb{D}_4$.

$$\begin{array}{cccccc}
C_1
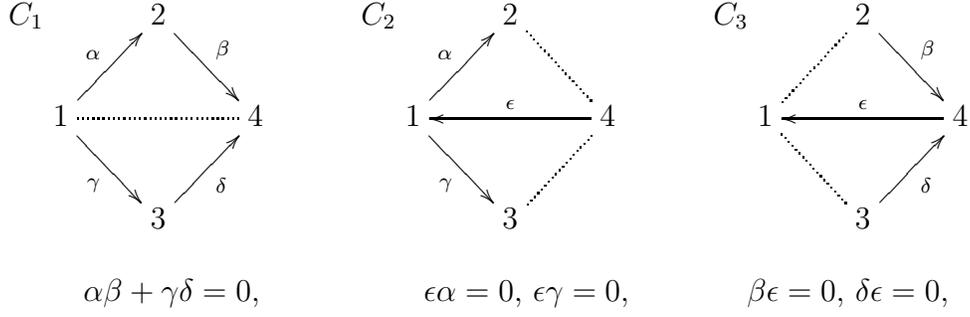
\xymatrix{&2\ar[rd]^\beta\\ 
1\ar[ru]^\alpha\ar[rd]_\gamma &&4\ar@{.}[ll]\\
&3\ar[ru]_\delta 
} \qquad
&C_2
\xymatrix{&2\ar@{.}[rd]\\ 
1\ar[ru]^\alpha\ar[rd]_\gamma &&4\ar[ll]_\epsilon\\
&3\ar@{.}[ru]
} \qquad
&C_3\xymatrix{&2\ar[rd]^\beta\\ 
1\ar@{.}[ru]\ar@{.}[rd] &&4\ar[ll]_\epsilon\\
&3\ar[ru]_\delta 
} \\

\\
\alpha\beta+\gamma\delta=0, &\,\epsilon\alpha=0, \, \epsilon\gamma=0,&\,\beta\epsilon=0,\,\delta\epsilon=0, \end{array}$$
The relation extension of each of these tilted algebras is the  following cluster-tilted algebra $B$.
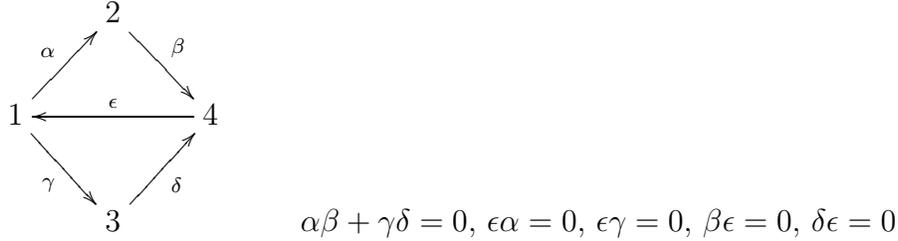
$$\begin{array}{cc}
 \xymatrix{&2\ar[rd]^\beta\\ 
1\ar[ru]^\alpha\ar[rd]_\gamma &&4\ar[ll]_\epsilon\\
&3\ar[ru]_\delta 
&&
\alpha\beta+\gamma\delta =0 ,\,\epsilon\alpha=0, \, \epsilon\gamma=0,\,\beta\epsilon=0,\,\delta\epsilon=0} \end{array}$$
We will compute an injective resolution
for the simple $B$-module $1$.

The $B$-module
$1$ is induced from $C_2$. In $\textup{mod\,}C_2$ we have the injective resolution
\[\xymatrix{0\ar[r]& 1\ar[r] & I_{C_2}(1) \ar[r]&I_{C_2}(4) \ar[r]&0}\]
and since $\textup{id}_{C_2}\, 1 =1$, 
Remark~\ref{5.7} yields the following exact sequence in $\textup{mod}\,B$
\[\xymatrix{0\ar[r]& 1\ar[r] & I_{B}(1) \ar[r]&I_{B}(4) \ar[r]&{\begin{smallmatrix} 1\\2 \ 3 \end{smallmatrix}} \ar[r]&0 ,}\]

Now the module $M={ \begin{smallmatrix} 1\\2 \ 3 \end{smallmatrix}}$ is not induced from $C_2$ or $C_3$ but it is induced from $C_1$.
In $\textup{mod\,}C_1$ we have the injective resolution
\[\xymatrix{0\ar[r]& {\begin{smallmatrix}1\\2 \ 3 \end{smallmatrix}}\ar[r] & I_{C_1}(2)\oplus I_{C_1}(3)  \ar[r]&I_{C_1}(1) \ar[r]&0}\]
and since $\textup{id}_{C_2}\, { \begin{smallmatrix} 1\\2 \ 3 \end{smallmatrix}}=1$, 
we obtain the following exact sequence in $\textup{mod}\,B$
\[\xymatrix{0\ar[r]&{\begin{smallmatrix}  1\\2 \ 3 \end{smallmatrix}}\ar[r] & I_{B}(2)\oplus I_{B}(3) \ar[r]&I_{B}(1) \ar[r]&{\begin{array}{c} 4 \end{array}} \ar[r]&0 ,}\]
and again we have added the second cosyzygy to the right.

Now in $\textup{mod}\,C_2$, the module 4 is  injective, $\tau\Omega \,4 ={ \begin{smallmatrix}  1\\2 \ 3 \end{smallmatrix}}$ and $\hat{I} =I_{C_2}(2)\oplus I_{C_2}(3)$. Thus Corollary~\ref{cor 11} yields the following exact sequence in $\textup{mod}\,B$
\[\xymatrix{0\ar[r]& 4\ar[r] & I_{B}(4) \ar[r]& I_{B}(2)\oplus I_{B}(3)\ar[r]&1 \ar[r]&0,}\]
and our new cosyzygy is again the simple module 1. Thus we have computed a periodic injective resolution of $1$.

\bigskip 
This procedure will produce injective resolutions as long as in each step the new second cosyzygy is again induced from some tilted algebra. In particular, for representation finite cluster-tilted algebras, it will produce an injective resolution  for every module, because in this case every module is induced \cite{SS}.
\end{exmp}

\bibliographystyle{plain}

\begin{thebibliography}{SimSko3}
\bibitem[A]{A} C. Amiot, Cluster categories for algebras of global dimension 2 and quivers with potential, \emph{Ann. Inst. Fourier} {\bf 59} no 6, (2009), 2525--2590. 
\bibitem[ABS] {ABS}{I. Assem,
  T. Br\"ustle and R. Schiffler}, Cluster-tilted 
  algebras  as trivial extensions,  \emph{Bull. Lond. Math. Soc.\/} {\bf 40} (2008), 151--162.
\bibitem [ABS2]{ABS2}{I. Assem,
  T. Br\"ustle and R. Schiffler}, Cluster-tilted algebras and slices, \emph{J. of Algebra\/} {\bf 319} (2008), 3464--3479. 
\bibitem[ABS3] {ABS3}{I. Assem,
  T. Br\"ustle and R. Schiffler}, On the Galois covering of a cluster-tilted algebra, \emph{J. Pure Appl. Alg.\/} {\bf 213} (7) (2009) 1450--1463.
\bibitem [ABS4]{ABS4}{I. Assem,
  T. Br\"ustle and R. Schiffler}, Cluster-tilted algebras without clusters, \emph{J. Algebra} {\bf 324}, (2010), 2475--2502.
  
\bibitem[ABIS]{ABIS} I. Assem, J.C. Bustamante, K. Igusa and R. Schiffler, The first Hochschild cohomology group of a cluster-tilted algebra revisited, \textit{Int. J. Alg. Comput.\/}  \textbf{23} (2013), no. 4, 729--744.
  \bibitem[AGST]{AGST} I. Assem, M. A. Gatica, R. Schiffler, and R. Taillefer, Hochschild cohomology of relation extension algebras, J. Pure Appl. Alg. 220, 7, (2016), 2471--2499. 
  \bibitem[AM]{AM} I. Assem and N. Marmaridis, {Tilting modules and split-by-nilpotent extensions}, {\it Comm. Algebra\/} {\bf{26}} (1998), 1547--1555.

 \bibitem [AR]{AR} I. Assem and M. J. Redondo, The first Hochschild cohomology group of a schurian cluster-tilted algebra, \emph{Manuscripta Math.} {\bf 128} (2009), no. 3, 373--388. 

\bibitem[ARS]{ARS} I. Assem,  M. J. Redondo and R. Schiffler, On the first Hochschild cohomology group of a cluster-tilted algebra, {\em Algebr. Represent. Theory\/} {\bf 18} (2015), no. 6, 1547--1576.

\bibitem[ASS]{ASS} I. Assem, D. Simson and A. Skowro\`nski,\emph{ Elements of the Representation Theory of Associative Algebras, 1: Techniques of Representation Theory}, London Mathematical Society Student Texts 65, Cambridge University Press, 2006.
\bibitem[{AZ}]{AZ} I. Assem and D. Zacharia, Full embeddings of almost split sequences over split-by-nilpotent extensions, \emph{Coll. Math.\/} {\bf 81}, (1) (1999) 21--31.

\bibitem [BFPPT]{BFPPT}{M. Barot,
  E. Fernandez, I. Pratti, M. I. Platzeck and S. Trepode}, From iterated tilted to
  cluster-tilted algebras, {\em Adv. Math.\/} {\bf 223} (2010), no. 4, 1468--1494. 
  \bibitem [BBT]{BBT} L. Beaudet, T. Br\"ustle and G. Todorov, Projective dimension of modules over cluster-tilted algebras, {\em Algebr. Represent. Theory\/} {\bf 17} (2014), no. 6, 1797--1807. 
  \bibitem [BOW]{BOW} M. A. Bertani-\O kland, S. Oppermann and A Wr\r{a}lsen, Constructing tilted algebras from cluster-tilted algebras,  {\em J. Algebra\/} {\bf 323} (2010), no. 9, 2408--2428. 
 \bibitem[BMRRT]{BMRRT} A. B. Buan, R. Marsh, M. Reineke, I. Reiten and G. Todorov, \emph{Tilting theory and cluster combinatorics}, Adv. Math. {\bf{204}} (2006), no. 2, 572-618.
\bibitem[{BMR}]{BMR}  { A. B. Buan, R. Marsh and I. Reiten},
  Cluster-tilted algebras, \emph{Trans. Amer. Math. Soc.} {\bf 359}
  (2007),  no. 1, 323--332 (electronic). 
  
\bibitem[{BMR2}]{BMR2}  { A. B. Buan, R. Marsh and I. Reiten},
 Cluster-tilted algebras of finite representation type, \emph{
 J. Algebra}  {\bf 306}  (2006),  no. 2, 412--431.  
 
 \bibitem [BMR3]{BMR3}  { A. B. Buan, R. Marsh and I. Reiten},
Cluster mutation via quiver representations. 
\emph{Comment. Math. Helv.} {\bf 83} (2008), no. 1, 143--177. 
 
  \bibitem[{CC}]{CC}  { P. Caldero and F. Chapoton},
  Cluster algebras as Hall algebras of quiver representations, {\it Comment. Math. Helv.} {\bf 81} (2006), no. 3, 595--616. 
  
\bibitem[{CCS}]{CCS}  { P. Caldero, F. Chapoton and
R. Schiffler}, Quivers with relations arising from clusters ($A_n$
case), \emph{Trans. Amer. Math. Soc.} {\bf 358} (2006), no. 3, 1347--1364. 

\bibitem [CCS2]{CCS2}  { P. Caldero, F. Chapoton and
R. Schiffler}, Quivers with relations and cluster tilted algebras,
  \emph{Algebr. and Represent. Theory\/} {\bf 9},  (2006),
  no. 4, 359--376.
\bibitem[CK]{CK}  {P. Caldero and B. Keller}, From triangulated
categories to
  cluster algebras, \emph{Invent. Math.} {\bf 172} (2008), 169--211.
\bibitem[DS]{DS} L. David-Roesler and R. Schiffler,
Algebras from surfaces without punctures. {\em J. Algebra\/} {\bf 350} (2012), 218--244.

\bibitem[DWZ]{DWZ} H. Derksen, J. Weyman and A. Zelevinsky, Quivers with potentials and their representations. I. Mutations. {\it Selecta Math.\/} (N.S.) {\bf 14} (2008), no. 1, 59--119.
\bibitem[FZ]{FZ} S. Fomin and A. Zelevinsky, Cluster algebras
I: Foundations, \emph{J. Amer. Math. Soc.} \bf 15 \rm (2002),
497--529.

\bibitem[{K}]{K}  {B. Keller}, On triangulated orbit categories,
  \emph{Documenta Math.} {\bf 10} (2005), 551--581.
 \bibitem[KR]{KR} B. Keller and I. Reiten, {Cluster-tilted algebras are Gorenstein and stably Calabi-Yau}, {\em Adv. Math.\/} {\bf 211} (2007), no. 1, 123--151.  
 \bibitem[L]{L} S. Ladkani, Hochschild cohomology of the cluster-tilted algebras of finite representation type, preprint, {\tt arXiv:1205.0799.}
 \bibitem[P]{P} P.G. Plamondon, Cluster algebras via cluster categories with infinite-dimensional morphism spaces. {\it Compos. Math.\/} {\bf 147} (2011), no. 6, 1921--1954.
 \bibitem[S]{S} R. Schiffler, \emph{Quiver Representations\/}, CMS Books in Mathematics, Springer  International Publishing, 2014.
 \bibitem[ScSe]{SS}  R. Schiffler and K. Serhiyenko, Induced and coinduced modules over cluster-tilted algebras, preprint, {\tt arXiv:1410.1732.}
 
\end{thebibliography}

\end{document}